\documentclass[runningheads,a4paper]{elsarticle}
\usepackage{amsmath,amsfonts,amsthm}
\usepackage{amssymb,amscd}
\usepackage{graphicx}

\def\muntz{M\"untz }
\def\divdif{\mathord\kern.43em\vrule width.6pt height5.6pt depth-.28pt \kern-.43em\Delta}

\newtheorem{theorem}{Theorem}
\newtheorem{lemma}{Lemma}
\newtheorem{proposition}{Proposition}
\newtheorem{definition}{Definition}
\newtheorem{corollary}{Corollary}

\newdefinition{remark}{Remark}

\begin{document}

\title{A \muntz Type Theorem for a Family of Corner Cutting Schemes}

\author[rvt,els]{Rachid Ait-Haddou\corref{cor1}}
\ead{rachid@bpe.es.osaka-u.ac.jp}
\author[focal]{Yusuke Sakane}
\author[rvt,els]{Taishin Nomura}

\cortext[cor1]{Corresponding author}

\address[rvt]{The Center of Advanced Medical Engineering and Informatics,
Osaka University, 560-8531 Osaka, Japan}
\address[focal]{Department of Pure and Applied Mathematics,
Graduate School of Information Science and Technology,
Osaka University, 560-0043 Osaka, Japan}
\address[els]{Department of Mechanical Science and Bioengineering
Graduate School of Engineering Science,
Osaka University, 560-8531 Osaka, Japan}

\begin{abstract}
By identifying a family of corner cutting schemes as 
a dimension elevation process of Gelfond-B\'ezier curves,
we give a \muntz type condition for the convergence of 
the generated control polygons to the underlying curve.
The surprising emergence of the \muntz condition in the 
problem raises the question of a possible connection 
between the density questions of nested Chebyshev 
spaces and the convergence of the corresponding 
dimension elevation algorithms.  
\end{abstract}      

\maketitle
%%%%%%%%%%%%%%%%%%%%%%%%%%%%%%%%%%%%%%%%%%%%%%%%%%%%%%%%%%%%%%%%%%%%%%%%%%%%%%%%%%%%%%%%%%%%%%%%%%
%%%%%%%%%%%%%%%%%%%%%%%%%%%%%%%%%%%%%%%%%%%%%%%%%%%%%%%%%%%%%%%%%%%%%%%%%%%%%%%%%%%%%%%%%%%%%%%%%%
\section{Introduction}
Let $n$ be a fixed positive integer and let 
$0 < r_1 < r_2 < ... r_n < r_{n+1}...< {r_m}<...$ be an 
infinite strictly increasing sequence of positive real numbers.
Given a polygon $(P_0,P_1,...,P_n)$ in $\mathbb{R}^s, s \geq 1$,
we apply the following corner cutting scheme :
For $i=0,1,...,n$, we set $P_{i}^{0} =P_{i}$ and for
$j=1,2,...$, we construct iteratively new polygons 
$(P_0^{j},P_1^{j},....,P_{n+j}^j)$ using the inductive rule  
\begin{equation}\label{initial}
P_{0}^{j} =P_{0}^{j-1} \quad  P_{n+j}^{j} = P_{n+j-1}^{j-1}
\end{equation}
and for $i=1,...,n+j-1$
\begin{equation}\label{cornercutting} 
P_{i}^{j} = \frac{r_{i}}{r_{n+j}} P_{i-1}^{j-1} +
\left( 1-\frac{r_{i}}{r_{n+j}} \right) P_{i}^{j-1}
\end{equation}
%%%%%%%%%%%%%%%%%%%%%%%%%%%%%%%%%%%%%%%
%%%%%%%%%%%%Figure 1 %%%%%%%%%%%%%%%%%% 
\begin{figure}
\hskip 1.5 cm
\includegraphics[height=5.cm]{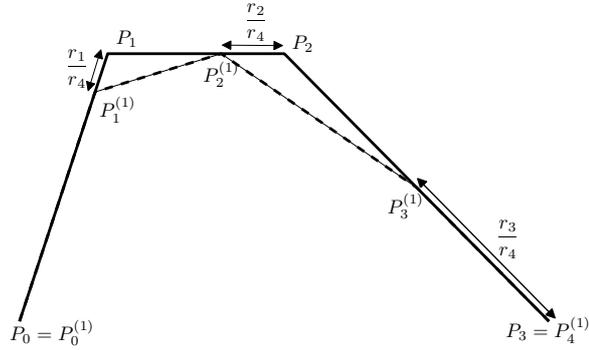}
\caption{The first iteration of the corner cutting scheme (\ref{initial})
and (\ref{cornercutting}) for the parameters $n =3, r_1 = 2, r_2 = 5, r_3 = 7$
and $r_4 = 14$.}
\label{fig:figure1}
\end{figure}
%%%%%%%%%%%%%%%%%%%%%%%%%%%%%%%%%%%%%%%
%%%%%%%%%%%%%%%%%%%%%%%%%%%%%%%%%%%%%%%
Figure \ref{fig:figure1} shows the first iteration of the corner
cutting scheme on a planar polygon $(P_0,P_1,P_2,P_3)$
for $n=3$ and positive real numbers $0 < r_1 < r_2 < r_3 < r_4$.  
In the case the real numbers $r_i$ are given by $r_i = i$ 
for every index $i$, then we recognize the degree elevation 
algorithm of B\'ezier curves, and in which it is well known 
that the control polygons of the elevated degree converges to 
the underlying B\'ezier curve.
Consider, now, the case in which $r_i = i$ for $i=1,...,n$ and 
$r_i = 2i$ for $i > n$. Figure \ref{fig:figure2} (left) shows 
the generated polygons from the scheme (\ref{initial}) and 
(\ref{cornercutting}) from four iterations, while Figure
\ref{fig:figure2} (right) shows the generated polygons
from 100 iterations. The figure suggests the convergence of
the generated polygons to the B\'ezier curve with 
control points $(P_0,P_1,...P_n)$. Consider, now, the case 
in which $r_i = i$ for $i=1,...,n$, while $r_i = i^2$ for $i >n$.
Figure \ref{fig:figure3} (left) shows the generated polygons 
from four iterations, while Figure \ref{fig:figure3} (right)
shows the obtained polygons after 100 iterations. It is clear
from the figure that the limiting polygon does not converge
to the B\'ezier curve with control points $(P_0,P_1,...,P_n)$.
As we will exhibit in this work, the main difference between 
the example of Figure \ref{fig:figure2} and the one of Figure
\ref{fig:figure3} is the fact that in the former we have 
$\sum_{i=1}^{\infty} 1/r_i  = \infty$, while in the latter 
we have $\sum_{i=1}^{\infty} 1/r_i  < \infty$.
%%%%%%%%%%%%%%%%%%%%%%%%%%%%%%%%%%%%%%%
%%%%%%%%%%%%Figure 2 %%%%%%%%%%%%%%%%%% 
\begin{figure}
\vskip 0.5cm {
\includegraphics[height=3.6cm]{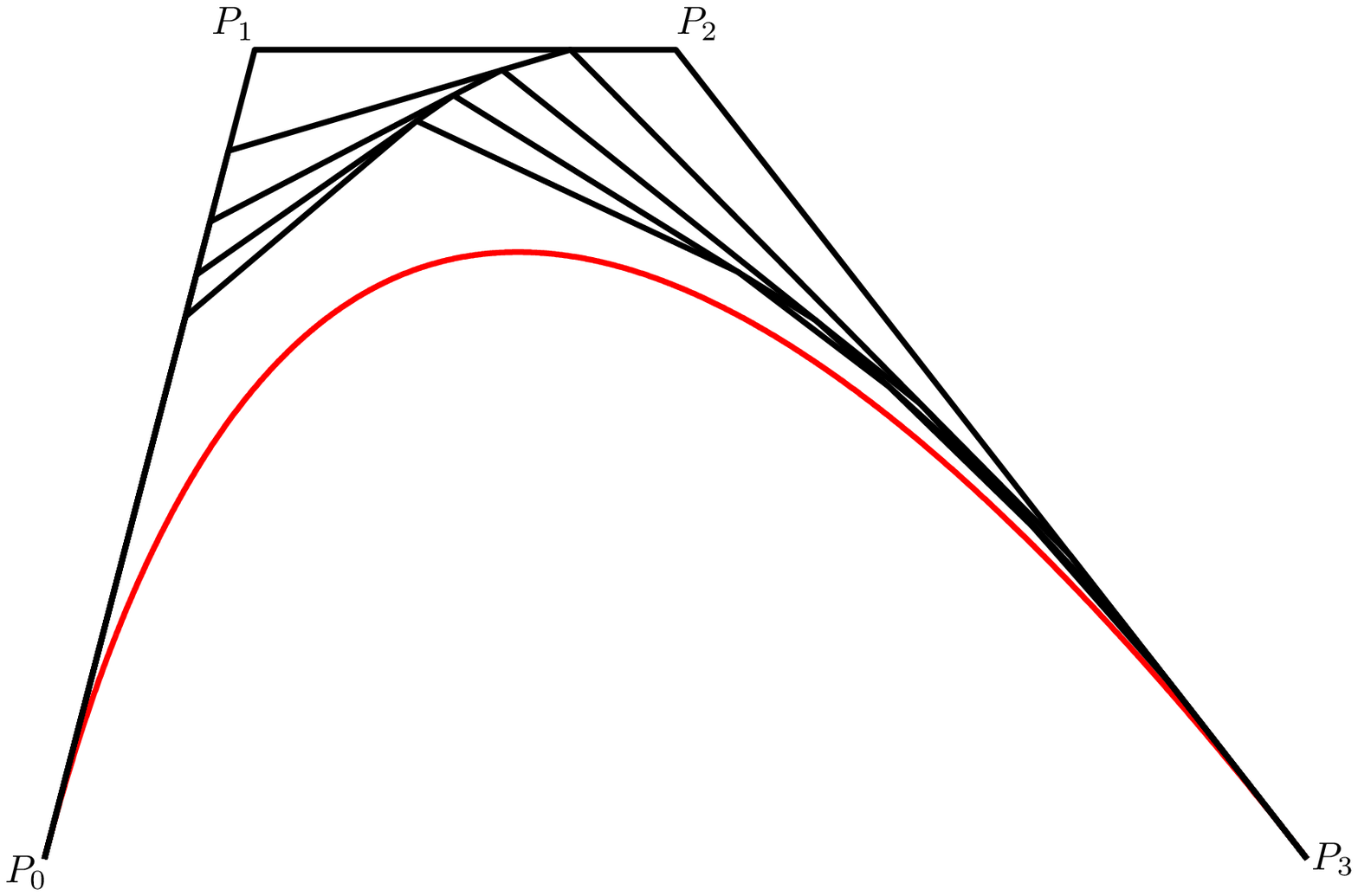}
\hskip 1.2 cm
\includegraphics[height=3.6cm]{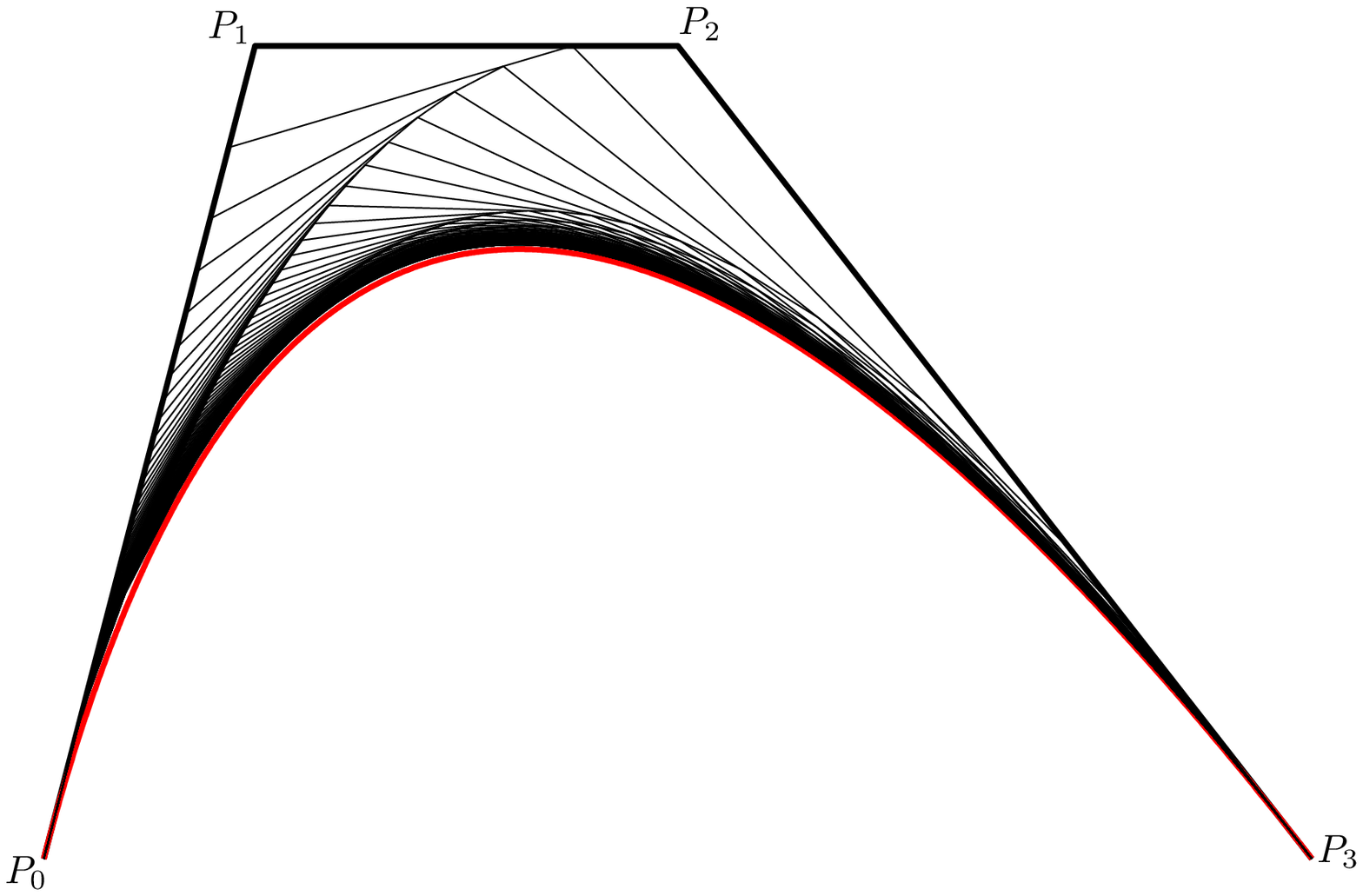}}
\caption{The sequence of polygons generated by the corner 
cutting scheme (\ref{initial}) and (\ref{cornercutting})
and parameters $n=3, r_1=1, r_2=2, r_3 =3$ and $r_j = 2j$ for $j \geq 4$.
(left, four iterations of the scheme; right, 100 iterations of the scheme). 
The red curve is the B\'ezier curve associated with the control polygon 
$(P_0,P_1,P_2,P_3)$.}
\label{fig:figure2}
\end{figure}
%%%%%%%%%%%%%%%%%%%%%%%%%%%%%%%%%%%%%%%
%%%%%%%%%%%%%%%%%%%%%%%%%%%%%%%%%%%%%%%
We will show, as a particular case of our main result, that 
if $r_i = i$ for $i=1,...,n$, and $\lim_{s\to\infty} r_s = \infty$,
then the limiting polygon generated from the corner cutting scheme
(\ref{initial}) and (\ref{cornercutting}) converges to the 
B\'ezier curve with control points $(P_0,P_1,...,P_n)$ 
if and only if the real numbers $r_i$ satisfy 
$\sum_{i=1}^{\infty} 1/r_i  = \infty$.
The emergence of the limiting polygon as a B\'ezier curve 
in the case $r_i = i$ for $i=1,...,n$ can be hinted to as follows : 
the linear space formed by the monomials with exponents the numbers
$r_i = i$ for $i=1,...,n$ and extended by a constant is given 
by $E = span(1,t,t^2,...,t^n)$; which is the linear space of polynomial
of degree $n$. The space $E$ has a special basis (the Bernstein basis)
in which the notion of B\'ezier curve can be defined. 
For general real numbers $r_i, i=1,...,n$, and imitating the previous 
construction, we obtain the \muntz space $F = span(1,t^{r_1},t^{r_2},
...,t^{r_n})$. The linear space $F$ also possess a special basis
(the Gelfond-Bernstein basis) first defined by Hirschman and Widder
\cite{hirschman} and extended by Gelfond \cite{gelfond}, which is in
a certain sense a generalization of the Bernstein basis to 
the \muntz space $F$ (in the case $r_i = i, i=1,...,n$, the 
Gelfond-Bernstein basis coincide with the Bernstein basis).
Using the Gelfond-Bernstein basis, we can canonically define the notion
of Gelfond-B\'ezier curve with control points $(P_0,P_1,...,P_n)$. 
Now, consider, for example, the limiting polygon of the corner 
cutting scheme (\ref{initial}) and (\ref{cornercutting})
for the case $n=3$ and in which $r_1=2, r_2 = 4, r_3 =5$ and 
$r_i = 2i$ for $i >3$.
%%%%%%%%%%%%%%%%%%%%%%%%%%%%%%%%%%%%%%%
%%%%%%%%%%%%Figure 3 %%%%%%%%%%%%%%%%%% 
\begin{figure}
\includegraphics[height=3.7cm]{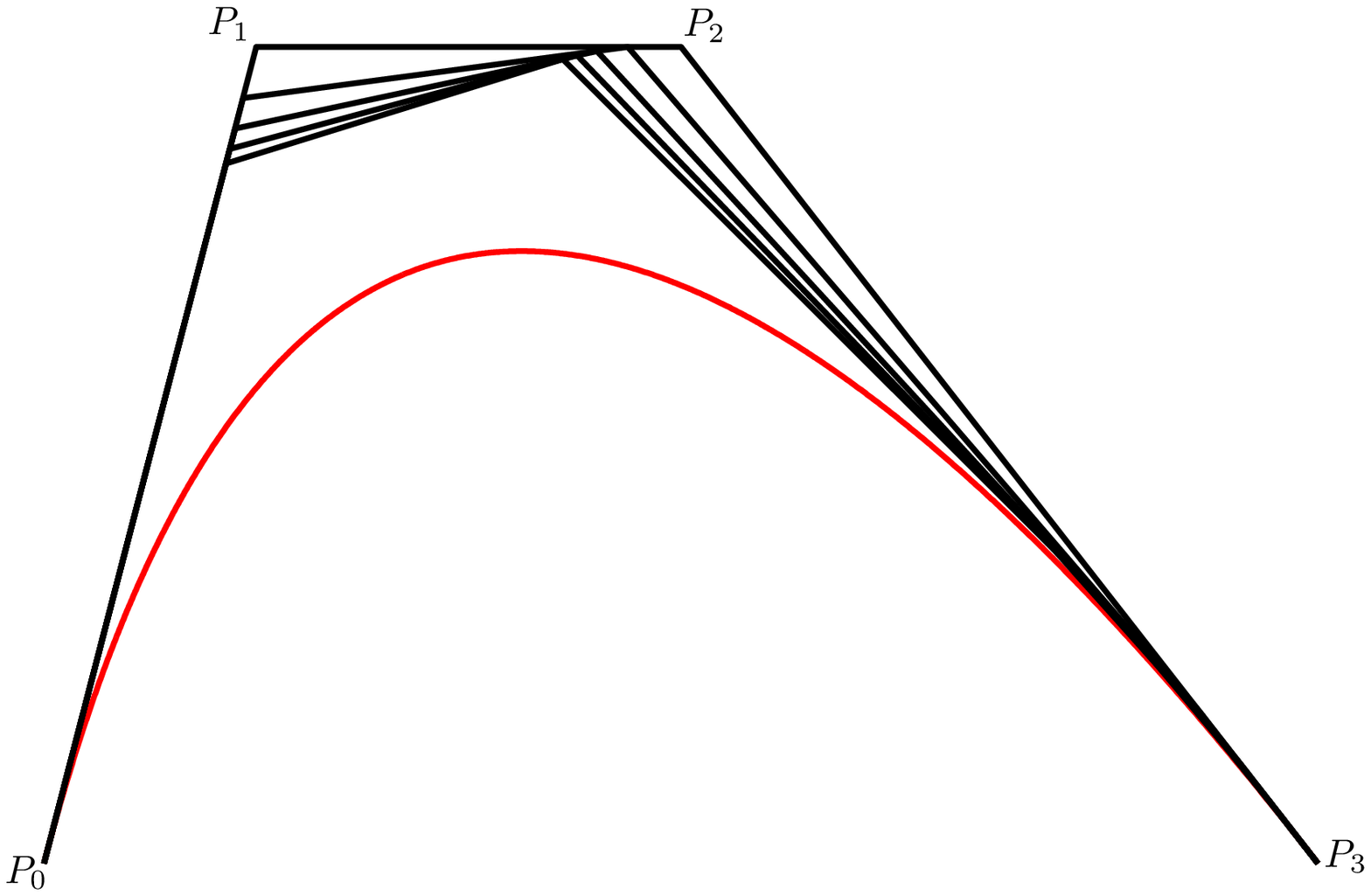}
\hskip 1.2 cm
\includegraphics[height=3.7cm]{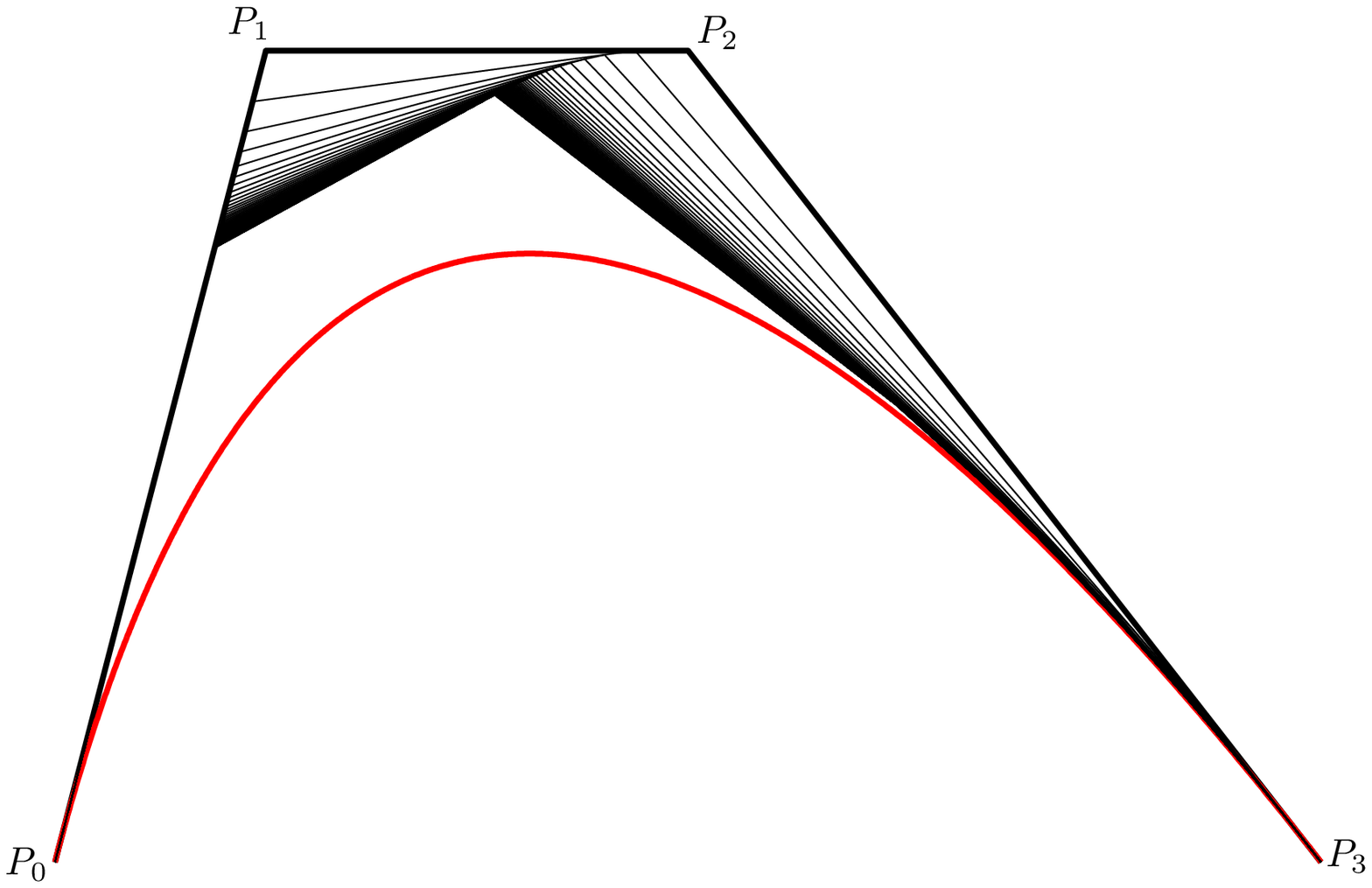}
\caption{The sequence of polygons generated by the corner 
cutting scheme (\ref{initial}) and (\ref{cornercutting})
and parameters $n=3, r_1=1, r_2=2, r_3 =3$ and $r_j = j^2$ for $j \geq 4$.
(left, four iterations of the scheme; right, 100 iterations of the scheme). 
The red curve is the B\'ezier curve associated with the control polygon 
$(P_0,P_1,P_2,P_3)$.}
\label{fig:figure3}
\end{figure}
%%%%%%%%%%%%%%%%%%%%%%%%%%%%%%%%%%%%%%%
%%%%%%%%%%%%%%%%%%%%%%%%%%%%%%%%%%%%%%%
Figure \ref{fig:figure4} shows the generated polygons from 100 iterations and 
also shows the Gelfond-B\'ezier curve associated with the \muntz space
$F = span(1,t^{r_1},t^{r_2},t^{r_3}) = span(1,t^2,t^4,t^5)$ and 
control polygon $(P_0,P_1,P_2,P_3)$. The figure suggests that the limiting 
polygon converges to the Gelfond-B\'ezier curve. Therefore,  
The main objective of this paper is to, effectively, prove the following 
%%%%%%%%%%%%%%%%%%
%%%%%%%%%%%%%%%%%%
\begin{theorem}\label{muntzelevation}
Let $n$ be a fixed positive integer and let 
$0 < r_1 < r_2 < ... r_n<r_{n+1}< ...<{r_m}<...$ be an 
infinite strictly increasing sequence of positive real numbers such that 
$\lim_{s\to\infty} r_{s} = \infty$. Then the limiting polygon
generated from a polygon $(P_0,P_1,...,P_n)$ in 
$\mathbb{R}^s, s \geq 1$ using the corner cutting scheme 
(\ref{initial}) and (\ref{cornercutting}) converges (pointwise 
and uniformly) to the Gelfond-B\'ezier curve associated with the \muntz space 
$(1,t^{r_1},t^{r_2},...,t^{r_n})$ and control polygon 
$(P_0,P_1,...,P_{n})$ if and only if the real numbers $r_i$ satisfy
the condition 
\begin{equation}\label{muntzcondition}
\sum_{i=1}^{\infty} \frac{1}{r_i} = \infty
\end{equation}
\end{theorem}
%%%%%%%%%%%%%%%%%%%%%%%%%%%%%%%%%%%%%%%
%%%%%%%%%%%%Figure 4 %%%%%%%%%%%%%%%%%% 
\begin{figure}
\vskip 0.1 cm
\hskip 2cm
\includegraphics[height=4.8cm]{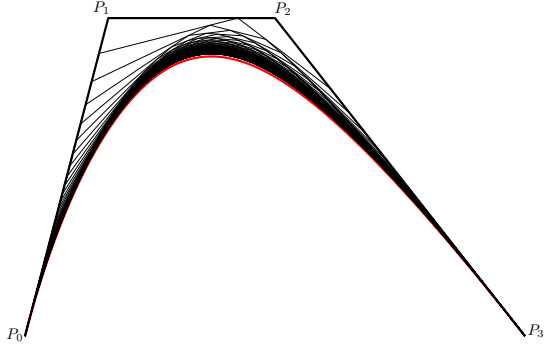}
\caption{The sequence of polygons generated from 100 iterations of 
the corner cutting scheme (\ref{initial}) and (\ref{cornercutting})
and parameters $n=3, r_1=2, r_2=4, r_3 =14$ and $r_j = 2j+10$ for $j \geq 4$. 
The red curve is the Gelfond-B\'ezier curve associated with the \muntz space 
$span(1,t^2,t^4,t^14)$ and control polygon $(P_0,P_1,P_2,P_3)$}
\label{fig:figure4}
\end{figure}
%%%%%%%%%%%%%%%%%%%%%%%%%%%%%%%%%%%%%%%
%%%%%%%%%%%%%%%%%%%%%%%%%%%%%%%%%%%%%%%
Let us contrast our theorem with the celebrated original 
\muntz theorem on the density of \muntz spaces \cite{muntz} 

\begin{theorem} {\bf{(\muntz Theorem)}}\label{muntztheorem}
Let $0 < r_1 < r_2 < ... r_n < r_{n+1}< ...<{r_m}<...$ be an
infinite strictly increasing sequence of positive real numbers such that 
$\lim_{s\to\infty} r_{s} = \infty$. The \muntz space 
$span(1,t^{r_1},...,t^{r_m},...)$ is a dense subset of $C([0,1])$ 
(the linear space of continuous functions on $[0,1]$ endowed with the uniform 
norm) if and only if 
\begin{equation*}
\sum_{i=1}^{\infty} \frac{1}{r_i} = \infty
\end{equation*}
\end{theorem}
The emergence of the \muntz condition (\ref{muntzcondition})
in both of Theorem \ref{muntzelevation} and Theorem \ref{muntztheorem}
is rather surprising and may suggest a deep relation between the 
problem of density in \muntz spaces and the convergence of corner cutting 
schemes. In fact, as we will show in section 2, the corner cutting scheme
(\ref{initial}) and (\ref{cornercutting}) can be interpreted as a dimension
elevation algorithm of Gelfond-B\'ezier curves. Therefore, Theorem 
\ref{muntzelevation} can be restated as claiming that under the condition 
that the sequence $0 < r_1 < r_2 ...< r_n <...$ satisfies 
$\lim_{s\to\infty} = \infty$, the density of the \muntz space 
$span(1,t^{r_1},...,t^{r_m},...)$ is equivalent to the convergence 
of the dimension elevation algorithm of Gelfond-B\'ezier curves 
to the underlying curve. We can push this analogy even further
as follows : It has been shown in \cite{ait-haddou1} that 
the Gelfond-Bernstein bases are limit of the Chebyshev-Bernstein 
bases of \muntz spaces over an interval $[a,1]$ as $a$ goes to zero.
From this property, it is not hard to show that the conditions of Theorem 
\ref{muntzelevation} are sufficient for the convergence of the dimension
elevation algorithm of a Chebyshev-B\'ezier curve in \muntz spaces to the 
underlying curve. As the Chebyshev-Bernstein bases over 
an interval $[a,b]$ can be defined for any linear space 
$E = span(1,u_1,...,u_m)$ such that the space 
$DE = span(u'_1,...,u'_m)$ is an extended Chebyshev space of order $m$ over 
the interval $[a,b]$ \cite{pottmann,mazure1}, we can ask for the following
more general question : Let $n$ be a fixed positive integer and let 
$u_1,u_2,...,u_n,...,u_m,...$ be an infinite sequence of $C^{\infty}$
functions over an interval $[a,b]$ such that for every $k \geq 1$, 
the space $E_k =span(1,u_1,u_2,...,u_k)$ is such that 
$DE_k = span(u'_1,u'_2...,u'_k)$ is an extended Chebyshev space of order $k$
over the interval $[a,b]$. For any function $F \in span(1,u_1,u_2,...,u_n)$ with 
control polygon $(P_0,...,P_n)$ over the interval $[a,b]$, 
we can define the control polygons of the dimension 
elevation algorithm \cite{mazure2} with respect to the nested spaces 
$E_{n} \subset E_{n+1} \subset ....\subset E_{m} \subset ...$,
the question is then 
%%%%%%%%%%%%%%%%%%%%%%%%%%%%%%%%%%%%%%%
\setcounter{equation}{3}
\newcounter{tmp_eqNum}
\setcounter{tmp_eqNum}{\value{equation}}
\renewcommand{\theequation}{\Alph{equation}}
\setcounter{equation}{16}
\begin{equation}
\begin{split}
&\textit{Is there a connection between the density of the space } 
E_{\infty} = span(1,u_1, \\
& u_2,...,u_n,...,u_m,...) \textit{as a subset of C([a,b]) and the convergence of the }\\
&\textit{associated dimension elevation algorithm to the underlying curve ?} \\
\end{split}
\end{equation}
\renewcommand{\theequation}{\arabic{equation}}
\setcounter{equation}{\value{tmp_eqNum}}
%%%%%%%%%%%%%%%%%%%%%%%%%%%%%%%%%%%%%%%  
A hypothesis of equivalence is ruled out by the following fact: 
the condition $\lim_{s\to\infty} r_s = \infty$ can be dropped 
in \muntz theorem \ref{muntztheorem}, however, such condition is 
necessary in  Theorem \ref{muntzelevation}. For example, Figure 
\ref{fig:figure5} (left) shows the limiting polygon for the case 
$n = 3$, $r_1 = 1, r_2 = 2, r_3 = 3$ and 
$r_j = 5 - \frac{1}{j}$ for $j>3$. The limiting polygon 
does not converge to the B\'ezier curve with control polygon 
$(P_0,P_1,P_2,P_3)$. As it will be clear within this work, the main 
reason for the non-convergence of the dimension elevation algorithm 
to the underlying curve in this case is the fact that the set of 
control points $(\eta^{m}_i)_{0\leq i \leq m}$ of the function 
$t^{r_1}$ with respect to the \muntz space 
$span(1,t^{r_1},t^{r_2},...,t^{r_m})$ does not 
form a dense subset of the interval $[0,1]$ as $m$ goes to infinity.
It is interesting to note that for example when $n = 3$ and  
the real number $r_i$ are given by $r_1 = 1, r_2 = 2, r_3 = 3$ and 
$r_i = 50 -\frac{1}{i}$, the limiting polygon is very close to the 
underlying curve and yet does not converge to the curve, as shown in 
Figure \ref{fig:figure5} (right). 

Regarding question (Q), we conjecture the following 
scenario : 
%%%%%%%%%%%%%
\noindent{\it{If for any fixed positive integer $n$, the dimension elevation 
algorithm with respect to the nested spaces $E_{n} \subset E_{n+1} \subset ....
\subset E_{m} \subset ...$ over an interval $[a,b]$ converges to 
the underlying Chebyshev-B\'ezier curve then the space $E_\infty$ is dense as a subset 
of $C([a,b])$.}}
\vskip 0.1 cm  
%%%%%%%%%%%%%%
%%%%%%%%%%%%%%%%%%%%%%%%%%%%%%%%%%%%%%%
%%%%%%%%%%%%Figure 5 %%%%%%%%%%%%%%%%%% 
\begin{figure}
\includegraphics[height=3.7cm]{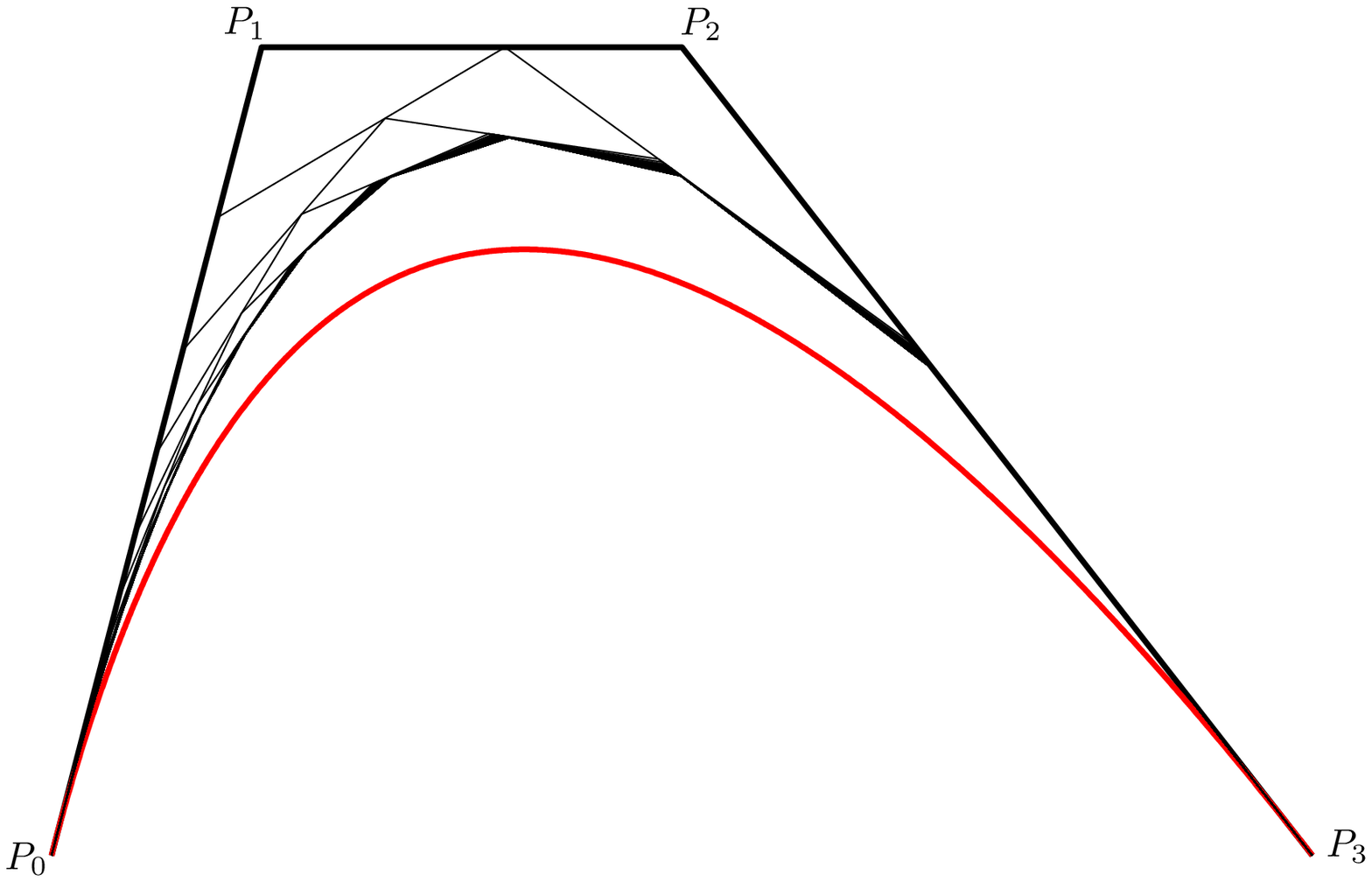}
\hskip 1.2 cm
\includegraphics[height=3.7cm]{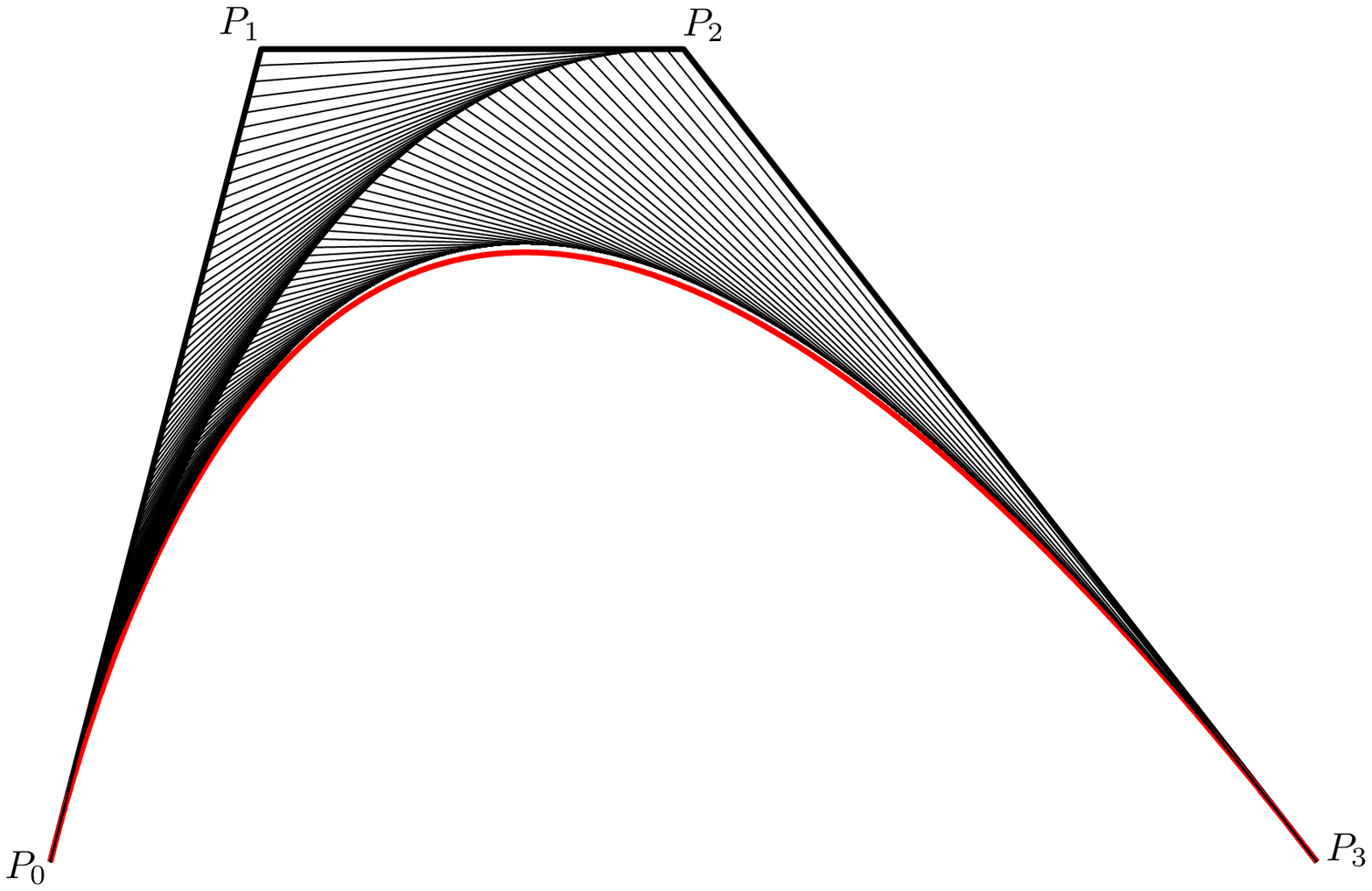}
\caption{The sequence of polygons generated from 100 iterations of 
the corner cutting scheme (\ref{initial}) and (\ref{cornercutting})
and parameters $n=3, r_1=1, r_2=2, r_3 =3$ and for the left figure 
we have $r_j = 5-\frac{1}{j}$ for $j \geq 4$
and for the right figure, $r_j = 50 - \frac{1}{j}$ for $j \geq 4$. 
The blue curve is the B\'ezier curve associated with 
the control polygon $(P_0,P_1,P_2,P_3)$.}
\label{fig:figure5}
\end{figure}
%%%%%%%%%%%%%%%%%%%%%%%%%%%%%%%%%%%%%%%
%%%%%%%%%%%%%%%%%%%%%%%%%%%%%%%%%%%%%%%

The proof of Theorem \ref{muntzelevation} consists
in first showing, in section 2, that the corner 
cutting scheme (\ref{initial}) and (\ref{cornercutting})
can be interpreted as a dimension elevation algorithm of 
the Gelfond-B\'ezier curves. This allows us, in section 3,
through a refinement of the elegant method of Prautzsch and 
Kobbelt \cite{kobbelt} to prove the theorem by induction on 
the fixed integer $n$.  
%%%%%%%%%%%%%%%%%%%%%%%%%%%%%%%%%%%%%%%%%%%%%%%%%%%%%%%%%%%%%%%%%%%%%%%%%%%%%%%%%%%%%%%%%%%%%%%%%%
%%%%%%%%%%%%%%%%%%%%%%%%%%%%%%%%%%%%%%%%%%%%%%%%%%%%%%%%%%%%%%%%%%%%%%%%%%%%%%%%%%%%%%%%%%%%%%%%%%
\section{Gelfond-B\'ezier curves}
Let $f$ be a smooth real function defined on an interval $I$. 
For any real numbers $x_0 \leq x_1 \leq ... \leq x_n$ in the interval $I$,
the divided difference $[x_0,...,x_n]f$ of the function $f$ supported
at the point $x_i, i=0,...,n$ is recursively defined by $[x_0]f = f(x_0)$ 
and  
\begin{equation}\label{dvdefinition}
[x_0,x_1,...,x_n]f = \frac{[x_1,...,x_n]f - [x_0,x_1,...,x_{n-1}]f}{x_n - x_0}
\quad \textnormal{if} \quad n>0.
\end{equation}
If some of the $x_i$ coincide, then the divided difference $[x_0,...,x_n]f$
is defined as the limit of (\ref{dvdefinition}) when the distance of the $x_i$ 
become arbitrary small. A simple inductive argument show that when the $x_i$ 
are pairwise distinct then we have 
{ \small {\begin{equation}\label{determinant}
[x_0,...,x_n]f = \sum^{n}_{i=0} \frac{f(x_i)}{\prod^{n}_{j=0, j \neq i}{(x_i - x_j)}}
= \frac{\left| \begin{array}{ccccc}
1 & x_{0} & \dots& x_{0}^{n-1} & f(x_0)  \\
1 & x_{1} & \dots& x_{1}^{n-1} & f(x_1)  \\
\dots & \dots & \dots & \dots            \\
1 & x_{n} & \dots& x_{n}^{n-1} & f(x_n) \\ 
\end{array} \right|} {V(x_0,x_1,...,x_n)},
\end{equation}}}
where $V(x_0,...,x_n)$ is the Vandermonde determinant. Note that by 
(\ref{dvdefinition}) the divided difference $[x_0,x_1,...,x_n]f$ 
is symmetric in the arguments $x_0,x_1,...x_n$.  
Consider, now, the function $f_{t}(x) = t^x$, where $t$ is viewed as 
a parameter. For a sequence $\Lambda= (0=r_0,r_1,...,r_n)$ of strictly
increasing real numbers, we denote by $E_{\Lambda}$ 
the \muntz space $E_{\Lambda} = span(t^{r_0},t^{r_1},...,t^{r_n})$.  
%%%%%%%%%%%%%%%%%%%%%%%%%%%%%%%%%%
%%%%%%%%%%%%%%%%%%%%%%%%%%%%%%%%%%
\begin{definition}
For a sequence $\Lambda=(0=r_0,r_1,...,r_n)$ of strictly
increasing real numbers, the Gelfond-Bernstein basis of the 
\muntz space $E_{\Lambda}$ with respect to the interval $[0,1]$
is defined by
$$
H^{n}_{k,\Lambda}(t) = (-1)^{n-k} r_{k+1}...r_{n} [r_k,...,r_n]f_{t}
\quad \textnormal{for} \quad k=0,...,n-1
$$
and 
$$
H^{n}_{n,\Lambda}(t) = t^{r_n}.
$$
\end{definition} 
%%%%%%%%%%%%%%%%%%%%%%%%%%%%%%%%%%
%%%%%%%%%%%%%%%%%%%%%%%%%%%%%%%%%%
The determinant representation of the divided differences (\ref{determinant})
shows that for $k=0,...,n-1$, the Gelfond-Bernstein basis can be expressed as 
\begin{equation}\label{lorentzdeterminant}
H^{n}_{k,\Lambda}(t) = \frac{r_{k+1}r_{k+2}...r_{n}}
{V(r_{k},r_{k+1},...,r_{n})} 
\left| \begin{array}{ccccc} 
t^{r_k} & 1 & r_{k} & \dots& r^{n-k-1}_{k}  \\
t^{r_{k+1}} & 1 & r_{k+1} & \dots& r^{n-k-1}_{k+1} \\
\dots & \dots & \dots & \dots \\
t^{r_{n}} & 1 & r_{n} & \dots& r^{n-k-1}_{n} \\ 
\end{array} \right|.
\end{equation}
Formula (\ref{lorentzdeterminant}) reiterate the fact that every function 
$H^{n}_{k,\Lambda},$ $k=0,...,n$ is an element of the space $E_{\Lambda}$.
The Gelfond-Bernstein basis possesses several properties that are similar to
the classical Bernstein basis over the interval $[0,1]$. For the sequence 
$\Lambda = (0,1,2,...,n)$, the Gelfond-Bernstein basis coincide with the 
classical Bernstein basis. Moreover, for any sequence 
$\Lambda=(0=r_0,r_1,...,r_n)$ of strictly
increasing real numbers, and for any $k = 0,...,n$, we have
\cite{ait-haddou1, lorentz} 
\begin{equation}\label{positivity} 
0 \leq H^{n}_{k,\Lambda} (t) \leq 1 
\quad \textnormal{for any}\quad t \in [0,1]
\end{equation}   
and for any $t \in [0,1]$, we have 
\begin{equation}\label{sumone}
\sum_{k=0}^n H^{n}_{k,\Lambda} (t) = 1.
\end{equation} 
Moreover, the Gelfond-Bernstein basis is totally positive on [0,1],
i.e. for any sequence $0 \leq t_0 < t_1 < ...< t_n \leq 1$, the matrix 
$(H^{n}_{k,\Lambda} (t_j))_{0 \leq k,j \leq n}$
is totally positive (i.e. all its minors are nonnegative). 
This property gives rise to the so-called variation diminishing
property of Gelfond-B\'ezier curve, i.e. given a Gelfond-B\'ezier curve 
$\Gamma$ with parametrization
\begin{equation}\label{gelfondbezier}
P(t) = \sum_{k=0}^{n} H^{n}_{k,\Lambda} (t) P_i
\quad \textnormal{with} \quad 
P_i \in \mathbb{R}^{s}, s \geq 1,
\end{equation}
the number of intersections of any hyperplane in $\mathbb{R}^s$
with $\Gamma$ does not exceed the number of intersection of the 
hyperplane with the control polygon $(P_0,P_1,...,P_n)$. 
Note also that for the Gelfond-B\'ezier curve in (\ref{gelfondbezier}),
we have $P(0) = P_0$ and $P(1) = P_n$. Figure \ref{fig:figure6} 
shows examples of Gelfond-B\'ezier curves associated with 
a single control polygon  $(P_0,P_1,P_2,P_3)$ and different sequences 
$\Lambda$. For a more thorough study of 
Gelfond-B\'ezier curves, we refer to \cite{ait-haddou1}.  
%%%%%%%%%%%%%%%%%%%%%%%%%%%%%%%%%%%%%%%
%%%%%%%%%%%%Figure 5 %%%%%%%%%%%%%%%%%% 
\begin{figure}
\hskip 2 cm
\includegraphics[height=5cm]{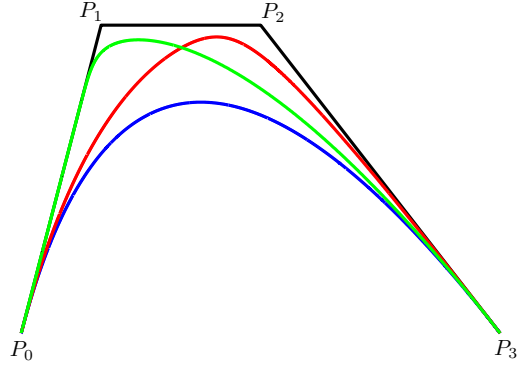}
\caption{Gelfond-B\'ezier curves associated with the control polygon 
$(P_0,P_1,P_2,P_3)$ and \muntz spaces : blue curve $span(1,t,t^2,t^3)$, 
red curve $span(1,t,t^2,t^{20})$, green curve $span(1,t^2,t^{50},t^{100})$.}
\label{fig:figure6}
\end{figure}
%%%%%%%%%%%%%%%%%%%%%%%%%%%%%%%%%%%%%%%
%%%%%%%%%%%%%%%%%%%%%%%%%%%%%%%%%%%%%%%

%%%%%%%%%%%%%%%%%%%%%%%%%%%%%%
%%%%%%%%%%%%%%%%%%%%%%%%%%%%%%
\begin{lemma}\label{iterationlemma}
Let $\Lambda_1 = (0=r_0,r_1,...,r_n)$ and $\Lambda_2 = (0=r_0,r_1,...,r_n,r_{n+1})$
be two sequences of strictly increasing real numbers.
Then, for $k=0,...,n$
\begin{equation}\label{iteration1}
H^{n}_{k,\Lambda_1}(t) = \frac{r_{n+1} - r_{k}}{r_{n+1}} H^{n+1}_{k,\Lambda_2}(t)
+ \frac{r_{k+1}}{r_{n+1}} H^{n+1}_{k+1,\Lambda_2}(t).
\end{equation}
\end{lemma}
%%%%%%%%%%%%%%%%%%%%%%%%%%%%%%
%%%%%%%%%%%%%%%%%%%%%%%%%%%%%%
\begin{proof}
From the definition of the Gelfond-Bernstein basis, the right hand 
side of equation (\ref{iteration1}) is given by ( for $k \leq n-1$)
\begin{equation}\label{iterationequation1}
(-1)^{n-k} r_k r_{k+1}...r_{n} \left( [r_{k+1},...,r_{n+1}]f_t - 
(r_{n+1} - r_{k}) [r_{k},...,r_{n+1}]f_t \right).
\end{equation}
From the definition of the divided difference, we have 
\begin{equation*}
[r_{k+1},...,r_{n+1}]f_t - [r_{k},...,r_{n}]f_t = (r_{n+1} - r_{k})
[r_{k},...,r_{n+1}]f_t.
\end{equation*}
Inserting the last equation into (\ref{iterationequation1})
conclude the proof of the lemma for $k \leq n-1$.
For $k=n$, the left hand side of (\ref{iteration1}) is equal to 
\begin{equation*}
t^{r_{n+1}} - (r_{n+1}-r_{n})[r_n,r_{n+1}]f_t = t^{r_n} = 
H^{n}_{n,\Lambda_1}(t).  
\end{equation*}
\end{proof}
%%%%%%%%%%%%%%%%%%%%%%%%%%
Let $\Lambda_1$ and $\Lambda_2$ be the two sequences
given in Lemma \ref{iterationlemma}, 
and let $P$ be an element of the \muntz 
space $E_{\Lambda_1}$. As $E_{\Lambda_1} 
\subset E_{\Lambda_2}$, the function $P$ can be 
expressed in both of the Gelfond-Bernstein bases associated
with the two spaces as  
\begin{equation}\label{expansion1}
P(t) = \sum_{k=0}^{n} H_{k,\Lambda_1}^{n}(t) P_{k} =
\sum_{k=0}^{n+1} H_{k,\Lambda_2}^{n+1}(t) \tilde{P}_{k}.
\end{equation}
Using Lemma \ref{iterationlemma} to detect the coefficients 
of $H_{k,\Lambda_2}^{n+1}(t)$ in the expansion 
(\ref{expansion1}), we readily find 
%%%%%%%%%%%%%%%%%%%%%%%%%%%%%%%%%%%%%%%%%%%%
%%%%%%%%%%%%%%%%%%%%%%%%%%%%%%%%%%%%%%%%%%%%
\begin{corollary}\label{elevationtheorem2}
The Gelfond-B\'ezier points $\tilde{P_k}$ in (\ref{expansion1})
are related to the Gelfond-B\'ezier points $P_{k}$ by the relations 
\begin{equation}\label{initial2}
\tilde{P}_0 = P_0, \quad \tilde{P}_{n+1} = P_{n},
\end{equation}
and for $k=1,2,...,n$
\begin{equation}\label{elevationequation}
\tilde{P}_k = \frac{r_{k}}{r_{n+1}} \; P_{k-1} + 
\left( 1 - \frac{r_k}{r_{n+1}} \right) \; P_{k}.
\end{equation}
\end{corollary}

Note that equations (\ref{initial2}) and (\ref{elevationequation})
describe the first iteration of the corner cutting scheme (\ref{initial})
and (\ref{cornercutting}). Therefore, the corner cutting scheme can 
be can interpreted as an iterative dimension elevation of 
the Gelfond-B\'ezier curve $P$ with respect to the nested 
\muntz spaces 
$E_{\Lambda_n} \subset E_{\Lambda_{n+1}} \subset ...E_{\Lambda_{m}} \subset ...$, 
where $\Lambda_k$ refers to the sequence $\Lambda_k = (0=r_0,r_1,...,r_k)$.   

For later use, we will need the following two lemmas, in which 
we omit the proofs as they can be readily obtained from the determinant 
representation (\ref{lorentzdeterminant}) of the Gelfond-Bernstein bases. 
 
\begin{lemma}\label{lemmat}
Let $H^{n}_{k,\Lambda}, k=0,...,n$ be the Gelfond-Bernstein basis associated
with the sequence $\Lambda = (0=r_0,r_1,r_2,...,r_n)$. Then, for $k=0,...,n$
we have 
\begin{equation*}
t H^{n}_{k,\Lambda}(t) = \frac{\prod_{j=k+1}^{n} r_j}
{\prod_{j=k+1}^{n} (r_j+1)}
H^{n+1}_{k+1,\Lambda_1}(t),
\end{equation*}
where $\Lambda_1 = (0=r_0,1,r_1+1,r_2+1,...,r_n+1)$
\end{lemma}

\begin{lemma}\label{pbarlemma}
Let $\alpha$ be a positive real number, and 
let $H^{n}_{k,\Lambda_1}, k=0,...,n$ (resp. $H^{n}_{k,\Lambda_2},
k=0,...,n$) be the Gelfond-Bernstein basis associated
with the sequence $\Lambda_1 = (0=r_0,r_1,r_2,...,r_n)$ (resp. 
$\Lambda_2 = (0=\alpha r_0,\alpha r_1,\alpha r_2,...,\alpha r_n)$.
Then, for $k=0,...,n$, we have
\begin{equation*}
H^{n}_{k,\Lambda_1}(t^{\alpha}) = H^{n}_{k,\Lambda_2}(t).
\end{equation*}
\end{lemma}

%%%%%%%%%%%%%%%%%%%%%%%%%%%%%%%%%%%%%%%%%%%%%%%%%%%%%%%%%%%%%%%%%%%%%%%%%%%%%%%
%%%%%%%%%%%%%%%%%%%%%%%%%%%%%%%%%%%%%%%%%%%%%%%%%%%%%%%%%%%%%%%%%%%%%%%%%%%%%%%
\section{The convergence of the dimension elevation algorithm}
The fundamental idea for the proof of Theorem \ref{muntzelevation}
is essentially simple, and can be viewed as a refinement of the 
method of Prautzsch and Kobbelt \cite{kobbelt}.
However, in practice, the simplicity of the idea is overshadowed
by the complexity of the technical details. Therefore, to exhibit
the fundamental idea of the proof, we will first use it for the proof
the classical fact that the control polygons of the degree elevation 
of a B\'ezier curve converge to the underlying curve.

Let $P$ be a polynomial of degree $n$ represented in the Bernstein basis 
,over the interval $[0,1]$, of degree $m > n$ as 
\begin{equation}\label{beziern}
P(t) = \sum_{i=0}^{m} b_{i}^{m} B^{m}_{i}(t).
\end{equation}
By induction on $n$, we will prove that 
\begin{equation}\label{convergence}
\max_{i}|b_{i}^{m} - P(\frac{i}{m})| = O(\frac{1}{m}).
\end{equation}
For $n \leq 1$, we have $b_{i}^{m} = P(i/m)$, since the 
Bernstein-B\'ezier representation has linear precision.
Now, let us assume that (\ref{convergence}) hold for polynomials
of degree $n-1$. Given a polynomial $P$ of degree $n$ 
with the Bernstein representation (\ref{beziern}),
we consider the polynomial $Q$ defined by 
$$
Q(t) = P(t) - \frac{t}{n}P'(t).
$$
The polynomial $Q$ is of degree $n-1$. Moreover, from (\ref{beziern}),
we have 
\begin{equation}\label{equationQ}
Q(t) =  \sum_{i=0}^{m} b_{i}^{m} B^{m}_{i}(t) - \frac{t}{n} 
\sum_{i=0}^{m-1} c_{i}^{m-1} B^{m-1}_{i}(t),
\end{equation}
where we have written the derivative $P'$ as 
\begin{equation*}
P'(t) = \sum_{i=0}^{m-1} c_{i}^{m-1} B^{m-1}_{i}(t).
\end{equation*} 
We can give explicit expressions for the coefficients $c_i^{m-1}$ but, 
as we will see, such expressions will not be needed.
Now, using the fact that
$$
t B_{i}^{m-1}(t) = \frac{i+1}{m} B_{i+1}^{m}(t)
$$
we obtain from (\ref{equationQ})
\begin{equation*}
Q(t) =  b_0^{m} B^{m}_{0}(t) +  \sum_{i=1}^{m} \left( b_{i}^{m} - 
\frac{i}{m n} c_{i-1}^{m-1} \right) B^{m}_{i}(t). 
\end{equation*}  
The induction hypothesis on $Q$ shows that 
\begin{equation*}
\max_{i} \left| \left( P(\frac{i}{m}) - b_{i}^{m} \right) - \frac{i}{m n}
\left( P'(\frac{i}{m}) - c_{i-1}^{m-1} \right) \right| = O(\frac{1}{m}). 
\end{equation*}
The last equation leads to (using the fact that $i \leq m$) 
\begin{equation}\label{equationmax}
\max_{i} | P(\frac{i}{m}) - b_{i}^{m}| = O(\frac{1}{m}) + 
\frac{1}{n} \max_{i} | P'(\frac{i}{m}) - c_{i-1}^{m-1} |. 
\end{equation}
Now, the polynomial $P'$ is also of degree $n-1$ and therefore, 
we can apply the induction hypothesis on $P'$, namely we have 
\begin{equation*}
\max_{i} |P'(\frac{i}{m-1}) - c_{i}^{m-1}| = O(\frac{1}{m-1}).
\end{equation*}
We have 
\begin{equation*}
|P'(\frac{i}{m}) - c_{i-1}^{m-1}| \leq  |P'(\frac{i}{m}) - P'(\frac{i-1}{m-1})| 
+ |P'(\frac{i-1}{m-1}) - c_{i-1}^{m-1}|.  
\end{equation*}
Therefore, we have 
\begin{equation*}
\max_{i} |P'(\frac{i}{m}) - c_{i-1}^{m-1}| \leq |\frac{i}{m}
- \frac{i-1}{m-1}| E + O(\frac{1}{m}) = O(\frac{1}{m}),
\end{equation*}
where $E = \max_{[0,1]} |P''(x) |$. Inserting the last equation into 
(\ref{equationmax}) conclude the proof of (\ref{convergence}).

%%%%%
\vskip 0.5 cm 
{\bf{Notations:}}
In order to apply the previous idea to the case of arbitrary \muntz 
spaces, we will first set some notations. We define the following 
difference operator $\Delta$ acting on sequences as follows : 
If $\Lambda_m = (0=r_0,r_1,r_2,...,r_m)$ is a sequence of   
strictly increasing real numbers, then  
\begin{equation*}
\Delta\Lambda_m = (0=r_0,r_2-r_1,r_3-r_1,...,r_m-r_1).
\end{equation*} 
For $1 \leq k \leq m-1$, we define the sequence 
$\Delta^{k}\Lambda_m$ iteratively using the 
equation
\begin{equation*}
\Delta^{k}\Lambda_m = \Delta(\Delta^{k-1}\Lambda_m) 
\quad \textnormal{with} \quad
\Delta^0\Lambda_m = \Lambda_m  
\end{equation*}
Therefore, we have 
\begin{equation*}
\Delta^{k}\Lambda_m = (0=r_0,r_{k+1}-r_{k},r_{k+2}-r_{k},...,r_{m}-r_{k}).
\end{equation*}
Now, for $i=0,...,m-k$, we denote by $\eta_{i}^{(k)}(\Lambda_m)$
the $i$th control point of the function $t^{r_{k+1}-r_{k}}$ with respect
to the \muntz space $E_{\Delta^{k}\Lambda_m}$, namely, we have 
\begin{equation*}
\eta_{0}^{(k)}(\Lambda_m) = 0; \quad \eta_{m-k}^{(k)}(\Lambda_m) = 1
\end{equation*}
and for $i=1,...,m-k-1$
\begin{equation}\label{controlpoints}
\eta_{i}^{(k)}(\Lambda_m) = \prod_{j=i+k+1}^{m} 
\left( 1 - \frac{r_{k+1}-r_{k}}{r_{j}-r_{k}} \right)
\end{equation}
We will adopt the convention that if $i <0$, then 
$\eta_{i}^{(k)}(\Lambda_0) = 0$ and also write  
$\eta_{i}^{(0)}(\Lambda_m)$ simply as $\eta_{i}(\Lambda_m)$.
We have

%%%%%%%%%%%%%%%%%%%%%%%%%%%%
%%%%%%%%%%%%%%%%%%%%%%%%%%%%
\begin{theorem}\label{polartheorem}
Let $\Lambda_n = (0=r_0,r_1,....,r_n)$ be a sequence 
of strictly increasing real numbers and let 
$\Lambda_m = (0=r_0,r_1,...,r_n,....,r_m)$ be a longer
sequence of strictly increasing numbers. 
Let $P$ be an element of the \muntz space $E_{\Lambda_n}$ 
written in the Gelfond-Bernstein bases of $E_{\Lambda_n}$ 
and $E_{\Lambda_m}$ as 
\begin{equation*}
P(t) = \sum_{i=0}^{n} p_{i}
H_{i,\Lambda_n}^{n}(t) = \sum_{i=0}^{m} b^{m}_{i}
H_{i,\Lambda_m}^{m}(t).
\end{equation*}
Then, there exist $(n-1)$ constants $C_k$ depending only on the 
function $P$ and the finite parameters $r_1,...,r_n$, such that 
\begin{equation}\label{polarequation}
\left| P(\eta_{i}(\Lambda_m)^{1/r_1}) - 
b_{i}^{m}\right| \leq \sum_{k=0}^{n-2} C_k 
\left|\eta_{i-k}^{(k)}(\Lambda_m) - 
\left( \eta_{i-(k+1)}^{(k+1)}(\Lambda_m) \right)^
{\frac{r_{k+1}- r_k}{r_{k+2}-r_{k+1}}}\right|
\end{equation}
for all $i=0,...,m$. We adopt the convention that $\sum_{k=0}^{-1} = 0$.
\end{theorem}
%%%%%%%%%%%%%%%%%%%%%%%%%%%%
%%%%%%%%%%%%%%%%%%%%%%%%%%%%
\begin{proof}
We will proceed by induction on $n$. For $n=1$ and as $\eta_{i}(\Lambda_m)$ 
are the control points of the function $t^{r_1}$ with respect to 
the \muntz space $E_{\Lambda_m}$ over the interval $[0,1]$,
we have $P(\eta_{i}(\Lambda_m)^{1/r_1}) - b_{i}^{m} = 0$
and the conclusion follows.
Let us assume the claim of the theorem is true for any element of a 
\muntz space of order less or equal to $n-1$. Let 
$P$ be an element of the space $E_{\Lambda_n}$ written 
in the Gelfond-Bernstein bases of $E_{\Lambda_n}$ 
and $E_{\Lambda_m}$ as
\begin{equation}\label{gelfondform}
P(t) = \sum_{i=0}^{n} p_{i}
H_{i,\Lambda_n}^{n}(t) = \sum_{i=0}^{m} b^{m}_{i}
H_{i,\Lambda_m}^{m}(t).
\end{equation}
Let us denote by $\bar{\Lambda}_n$ and 
$\bar{\Lambda}_m$ the sequences  
$\bar{\Lambda}_n = (0=r_0,1,r_2/r_1,...,r_n/r_1)$
and $\bar{\Lambda}_m = (0=r_0,1,r_2/r_1,...,r_n/r_1,....,r_m/r_1)$. 
To the function $P$ in (\ref{gelfondform}), we associate the function
$\bar{P}$ in the space $E_{\bar{\Lambda}_n}$ defined as 
\begin{equation}\label{pbar}
\bar{P}(t) = \sum_{i=0}^{n} p_{i}
H_{i,\bar{\Lambda}_n}^{n}(t).
\end{equation}
As the corner cutting scheme (\ref{initial}) and (\ref{cornercutting})
associated with a sequence $S = (r_1,...,r_l)$ is invariant by a multiplication of 
every elements of $S$ by the same scalar, we necessarily have
\begin{equation*}
\bar{P}(t) = \sum_{i=0}^{m} b_{i}^{m}
H_{i,\bar{\Lambda}_m}^{m}(t).
\end{equation*}
Consider the following function $Q$ defined by  
$$
Q(t) = \bar{P}(t) - \frac{r_1}{r_n} t \bar{P}'(t).
$$
It can be readily checked that the function $Q$ is an element of 
the \muntz space of order $n-1$, $E_{\Pi_{n-1}}$ where 
$\Pi_{n-1} = (0=r_0,1,r_2/r_1,....,r_{n-1}/r_1)$. 
Therefore, we can apply the induction hypothesis to the function $Q$.
Before we apply such induction, let us first express the function $Q$
in the Gelfond-Bernstein basis $H_{i,\bar{\Lambda}_m}^{m}$.
We have 
\begin{equation}\label{equationQ2}
Q(t) =  \sum_{i=0}^{m} b^{m}_{i} H_{i,\bar{\Lambda}_m}^{m}(t) -
\frac{r_1}{r_n} t  \sum_{i=0}^{m-1} c^{m-1}_{i} H_{i,\Gamma_{m-1}}^{m-1}(t),   
\end{equation}
where we have denoted $\bar{P}'$ as 
\begin{equation}\label{derivative}
\bar{P}'(t) = \sum_{i=0}^{m-1} c^{m-1}_{i} H_{i,\Gamma_{m-1}}^{m-1}(t),
\end{equation}
where $\bar{H}_{i,\Gamma_{m-1}}^{m-1}$ is the Gelfond-Bernstein basis 
with respect to the sequence 
$\Gamma_{m-1} = (0=r_0,(r_2/r_1)-1,(r_3/r_1)-1,...,(r_m/r_1)-1)$.
From Lemma \ref{lemmat}, we have 
\begin{equation*}
t H_{i,\Gamma_{m-1}}^{m-1}(t) = \prod_{j=i+2}^{m} \left(1 - \frac{r_1}{r_j} \right) 
H_{i+1,\bar{\Lambda}_m}^{m}(t) = \eta_{i+1}(\bar{\Lambda}_m)
H_{i+1,\bar{\Lambda}_m}^{m}(t). 
\end{equation*}
Therefore, from (\ref{equationQ2}), we have
\begin{equation*}
Q(t) = b_{0}^{m} H_{0,\bar{\Lambda}_m}^{m}(t) + 
\sum_{i=1}^{m} \left( b_{i}^{m} - \frac{r_1}{r_n} \eta_{i}(\bar{\Lambda}_m)
c^{m-1}_{i-1} \right) H_{i,\bar{\Lambda}_m}^{m}(t).   
\end{equation*}
The last equation shows that we have applied upon $Q$ a dimension elevation
from the \muntz space associated with the sequence $\Pi_{n-1}$ to the \muntz space
associated with the sequence $\bar{\Lambda}_m$. Moreover, it can easily
be checked that $\eta_{i}^{(k)}(\bar{\Lambda}_m) = \eta_{i}^{(k)}(\Lambda_m)$
for any $i$ and $k$.
Therefore, the induction hypothesis and the expression of $Q$ show 
that there exist $(n-2)$ constant $C_k$ depending only on the polynomial
$Q$ and the parameters $r_1,...,r_{n-1}$ such that for $i=0,...,m$ we have  
\begin{equation*}
\begin{split}
& \left| \left( \bar{P}(\eta_{i}(\Lambda_m)) - b_{i}^{m}\right) 
- \frac{r_1}{r_n} \eta_{i}(\Lambda_m)
\left(\bar{P}'(\eta_{i}(\Lambda_m))) - c_{i-1}^{m-1}\right) \right| \leq  \\
& \sum_{k=0}^{n-3} C_k 
\left|\eta_{i-k}^{(k)}(\Lambda_m) - \left( \eta_{i-(k+1)}^{(k+1)}(\Lambda_m) \right)^
{\frac{r_{k+1}-r_{k}}{r_{k+2}-r_{k+1}}}\right|.
\end{split}
\end{equation*}
Therefore, we have 
\begin{equation}\label{equationconvergence}
\begin{split}
\left| \bar{P}(\eta_{i}(\Lambda_m)) - b_{i}^{m} \right| \leq 
& \sum_{k=0}^{n-3} C_k 
\left|\eta_{i-k}^{(k)}(\Lambda_m) - \left( \eta_{i-(k+1)}^{(k+1)}(\Lambda_m) \right)^
{\frac{r_{k+1}-r_k}{r_{k+2}-r_{k+1}}}\right| +  \\
& \left| \frac{r_1}{r_n} \eta_{i}(\Lambda_m)
\left(\bar{P}'(\eta_{i}(\Lambda_m)) - c_{i-1}^{m-1}\right) \right|. 
\end{split}
\end{equation}
Now the function $\bar{P}'$ is an element of the \muntz space of order $n-1$,
$E_{\Gamma_{n-1}}$, where 
$\Gamma_{n-1} = (0=r_0,(r_2/r_1)-1,(r_3/r_1)-1,...,(r_n/r_1)-1)$. 
The expression (\ref{derivative}), shows that we have applied upon $\bar{P}'$ 
a dimension elevation from the \muntz space associated with the sequence
$\Gamma_{n-1}$ to the \muntz space associated with the sequence $\Gamma_{m-1}$.
Therefore, again by the induction hypothesis, there exist $(n-2)$ constants $E_k$ 
depending only  on $\bar{P'}$ and the parameters $r_1,...,r_n$ such that
\begin{equation*}
\left| \bar{P}'(\eta_{i}(\Gamma_{m-1})^{\frac{r_1}{r_2 - r_1}}) 
- c_{i}^{m-1} \right| \leq 
\sum_{k=0}^{n-3} E_k 
\left|\eta_{i-k}^{(k)}(\Gamma_{m-1}) - \left( \eta_{i-(k+1)}^{(k+1)}
(\Gamma_{m-1}) \right)^
{\frac{r_{k+2}-r_{k+1}}{r_{k+3}-r_{k+2}}}\right| 
\end{equation*}
It can be easily shown that $\eta_{i}^{(k)}(\Gamma_{m-1}) = \eta_{i}^{(k+1)}(\Lambda_m)$.
Therefore, we have 
\begin{equation*}
\left| \bar{P}'(\eta_{i}^{(1)}(\Lambda_m)^{\frac{r_1}{r_2 - r_1}}) - c_{i}^{m-1} \right| \leq 
\sum_{k=0}^{n-3} E_k 
\left|\eta_{i-k}^{(k+1)}(\Lambda_m) - \left( \eta_{i-(k+1)}^{(k+2)}(\Lambda_m) \right)^
{\frac{r_{k+2}-r_{k+1}}{r_{k+3}-r_{k+2}}}\right|. 
\end{equation*}
Moreover, we have 
\begin{equation*}
\begin{split}
\left| \bar{P}'(\eta_{i}(\Lambda_m)) - c_{i-1}^{m-1} \right| 
& \leq \left| \bar{P}'(\eta_{i}(\Lambda_m)) -  \bar{P}'((\eta_{i-1}^{(1)}
(\Lambda_m))^{\frac{r_1}{r_2 - r_1}}) \right| + \\
& \quad \left| \bar{P}'((\eta_{i-1}^{(1)}(\Lambda_m))^{\frac{r_1}{r_2 - r_1}})
- c_{i-1}^{m-1} \right|.
\end{split}
\end{equation*} 
Therefore,
\begin{equation*}
\begin{split}
& \left| \bar{P}'(\eta_{i}(\Lambda_m)) - c_{i-1}^{m-1} \right| \leq
\left| \eta_{i}(\Lambda_m) -  (\eta_{i-1}^{(1)}(\Lambda_m))^{\frac{r_1}{r_2 - r_1}})
\right| C + \\
&  \sum_{k=1}^{n-2} E_{k-1} 
\left|\eta_{i-k}^{(k)}(\Lambda_m) - \left( \eta_{i-(k+1)}^{(k+1)}(\Lambda_m) \right)^
{\frac{r_{k+1}-r_{k}}{r_{k+2}-r_{k+1}}}\right|
\end{split} 
\end{equation*}
where $C = max_{[0,1]}|\bar{P}''(x)|$. 
Inserting the last inequality into (\ref{equationconvergence}) and 
using the obvious fact that $|(\eta_{i}(\Lambda_m))| \leq 1$, show that 
there exist $(n-2)$ constants $L_i$ depending only on the polynomial
$\bar{P}$ and the real values $r_1,...,r_n$ such that 
\begin{equation}\label{lastequation}
\left| \bar{P}(\eta_{i}(\Lambda_m)) - 
b_{i}^{m}\right| \leq \sum_{k=0}^{n-2} L_k 
\left|\eta_{i-k}^{(k)}(\Lambda_m) - 
\left( \eta_{i-(k+1)}^{(k+1)}(\Lambda_m) \right)^
{\frac{r_{k+1}-r_{k}}{r_{k+2}-r_{k+1}}}\right|.
\end{equation}
In view of Lemma \ref{pbarlemma} and equation (\ref{pbar}), we have 
\begin{equation*}
\bar{P}(\eta_{i}(\Lambda_m)) = P(\eta_{i}(\Lambda_m)^{1/r_1}).
\end{equation*}
Inserting the last equation into (\ref{lastequation}) conclude 
the proof of the theorem.
\end{proof}

The following lemma is implicit in \cite{lorentz}, and even
more explicit in \cite{aldaz}, as our hypothesis are different
from the ones taken in the latter and for the seek of completeness,
we will include it proof. 
%%%%%%%%%%%%%%%%%%%%%
%%%%%%%%%%%%%%%%%%%%%
\begin{lemma}\label{lorentzlemma}
Let $\gamma$ be a strictly positive and let $a_j$ and $b_j$, $j=1,2...$  be
sequences of real numbers in $]0,1[$ such that 
\begin{equation}\label{condition1}
\lim_{j\to\infty} \frac{\ln b_j}{\ln a_j} = \gamma.
\end{equation}
Define $A_{i}(m) = \prod_{j=i+1}^m a_j$ and $B_{i}(m) = \prod_{j=i+1}^m b_j$
$(i<m)$ and let us assume that for any fixed $i$, we have 
\begin{equation}\label{condition2}
\lim_{m \to \infty} A_{i}(m) = 0
\quad \textnormal{and} \quad
\lim_{m \to \infty} B_{i}(m) = 0.
\end{equation}
Then 
\begin{equation*}
\lim_{m\to\infty} (A_{i}(m)^{\gamma} - B_{i}(m)) = 0
\quad \textnormal{uniformly in} \quad i.
\end{equation*}
\end{lemma}
%%%%%%%%%%%%%%%%%%%%%
%%%%%%%%%%%%%%%%%%%%%
\begin{proof}
We should prove that for every $\epsilon_1 > 0$, there exist 
an $m_0$ such that for all $m \geq m_0$ we have
\begin{equation*} 
|A_{i}(m)^{\gamma} - B_{i}(m)| < \epsilon_1 
\quad \textnormal{for} \quad
i=1,2,...,m.
\end{equation*}
Let us fix an $\epsilon_1 >0$ and select an $\epsilon \in]0,1[$ such that 
\begin{equation*}
1 - \epsilon^{\epsilon} < \epsilon_1, 
\quad \epsilon^{\gamma} < \epsilon_1;
\quad \epsilon < \gamma \quad \textnormal{and} \quad  
\epsilon^{\gamma-\epsilon} < \epsilon_1.
\end{equation*}
Condition (\ref{condition1}) shows that there exists a $j_0$ such that for 
any $j \geq j_0$, we have 
\begin{equation}\label{logequation}
\gamma - \epsilon < \frac{log\, b_j}{log\, a_j} < \gamma + \epsilon. 
\end{equation}
Since $log\, a_j < 0$, (\ref{logequation}) imply that for any $j \geq j_0$, we have  
\begin{equation*}
a_j ^{\gamma+\epsilon} < b_j < a_j^{\gamma-\epsilon}.
\end{equation*}
The last equation shows that for any $j \geq j_0$ and for any 
$m \geq j_0$ we have 
\begin{equation*}
A_j(m)^{\gamma+\epsilon} \leq B_j(m) \leq A_j(m)^{\gamma-\epsilon}.
\end{equation*} 
We can rephrase the last assertion as follow : There exists an 
$m_0 = j_0$ such that for any $m \geq m_0$, we have
\begin{equation}\label{importantequation}
A_j(m)^{\gamma+\epsilon} \leq B_j(m) \leq A_j(m)^{\gamma-\epsilon}, 
\quad \textnormal{for} \quad j = j_0,j_0+1,...,m.
\end{equation} 
Let us fix $m \geq m_0=j_0$. If for a certain index $j \geq j_0$, we 
have $A_{j}(m)^{\gamma} \geq B_j(m)$ and $A_j(m) < \epsilon$, 
then, we have 
$$
0 \leq A_{j}(m)^{\gamma} - B_j(m) < \epsilon^{\gamma} < \epsilon_1.      
$$
If for a certain index $j \geq j_0$, we 
have $A_{j}(m)^{\gamma} \geq B_j(m)$ and $A_j(m) \geq \epsilon$, 
then from (\ref{importantequation}) and using the fact that $A_j(m) <1$, we have
$$
0 \leq A_{j}(m)^{\gamma} - B_j(m) \leq A_j(m)^{\gamma} - A_j(m)^{\gamma+\epsilon}
= A_j(m)^{\gamma} (1 - A_j(m)^{\epsilon}) < 1 - \epsilon^{\epsilon} < \epsilon_1.
$$ 
If for a certain index $j \geq j_0$, we have 
$A_{j}(m)^{\gamma} \leq B_j(m)$ and $A_j(m) < \epsilon$, 
then, from (\ref{importantequation}), we have
$$
0 \leq B_j(m) - A_j(m)^{\gamma} \leq A_{j}(m)^{\gamma-\epsilon} -A_{j}(m)^{\gamma}
\leq \epsilon^{\gamma-\epsilon} < \epsilon_1.
$$  
Finally, if for a certain index $j \geq j_0$, we have 
$A_{j}(m)^{\gamma} \leq B_j(m)$ and $A_j(m) \geq \epsilon$, 
then, from (\ref{importantequation}), we have 
$$
0 \leq B_j(m) - A_j(m)^{\gamma} \leq 1 - A_{j}^{\gamma}
\leq 1 - \epsilon^{\epsilon} < \epsilon_1.
$$  
As we have exhausted all the possible cases on the behavior of 
a pair of numbers $A_j(m)$ and $B_j(m)$ for a certain index $j \geq j_0$, 
the conclusion of theses cases show that for any $m \geq m_0=j_0$, we have
\begin{equation}\label{infiniteequation}
|A_j(m)^{\gamma} - B_j(m)| < \epsilon_1
\quad \textnormal{for} \quad j = j_0,j_0+1,...,m.
\end{equation}  
Now condition (\ref{condition2}), shows in particular that for any $j < j_0$, 
there exists an $M_0(j)$ such that for any $m \geq M_0(j)$, 
we have  
\begin{equation*}
|A_j(m)^{\gamma} - B_j(m)| < \epsilon_1.
\end{equation*}
As we have a finite set of $M_0(j), j=1,...,j_0-1$, if 
we denote by $L_0 = \max_{j=1,...,j_0-1}(M_0(j))$, then for any 
$m \geq L_0$, we have
\begin{equation}\label{finiteequation}
|A_j(m)^{\gamma} - B_j(m)| < \epsilon_1
\quad \textnormal{for} \quad j = 1,2,...,j_0-1.
\end{equation}    
Therefore, by taking $M_0 = max(m_0,L_0)$ and in view of 
(\ref{infiniteequation}) and (\ref{finiteequation}),
we have for any $m \geq M_0$
\begin{equation*}
|A_j(m)^{\gamma} - B_j(m)| < \epsilon_1
\quad \textnormal{for} \quad j = 1,2,...,m.
\end{equation*} 
\end{proof} 
From the last lemma, we can prove the following 
%%%%%%%%%%%%%%%%%%%%%%%%%%%%%%
%%%%%%%%%%%%%%%%%%%%%%%%%%%%%%
\begin{theorem}\label{intermediate}
Let $\Lambda_\infty = (0=r_0,r_1,...,r_n,....,r_m,...)$ be an infinite 
sequence of strictly increasing real numbers such that 
\begin{equation}\label{twoconditions}
\lim_{s\to \infty} r_s = \infty
\quad \textnormal{and} \quad
\sum_{i=1}^{\infty} \frac{1}{r_i} = \infty.
\end{equation}
For any integer $m$, we denote by $\Lambda_m$ the 
subsequence of $\Lambda_{\infty}$ given by 
$\Lambda_m = (0=r_0,r_1,...,r_m)$. Let $P$ be an element 
of the \muntz space $E_{\Lambda_n}$ written in the 
Gelfond-Bernstein bases of $E_{\Lambda_n}$ and $E_{\Lambda_m}$
($m \geq n$) as 
\begin{equation*}
P(t) = \sum_{i=0}^{n} p_{i}
H_{i,\Lambda_n}^{n}(t) = \sum_{i=0}^{m} b^{m}_{i}
H_{i,\Lambda_m}^{m}(t).
\end{equation*} 
Then
\begin{equation}\label{convergenceterm}
\lim_{m\to\infty} P(\eta_{i}(\Lambda_m)^{1/r_1}) - 
b_{i}^{m} = 0 
\quad \textnormal{uniformly in} \quad i.
\end{equation} 
\end{theorem}
%%%%%%%%%%%%%%%%%%%%%%%%%%%%%%
%%%%%%%%%%%%%%%%%%%%%%%%%%%%%%
\begin{proof}
In view of Theorem \ref{polartheorem}, we need only to show
that under the conditions (\ref{twoconditions}),
the right hand side of (\ref{polarequation}) converges 
to zero as $m$ goes to infinity, uniformly in $i$.
As we have finite terms in the sum in (\ref{polarequation}), 
we will only need to show that for any fixed $k$ such that 
$0\leq k \leq n-2$, we have
\begin{equation}\label{terms}
\lim_{m\to\infty} \eta_{i-k}^{(k)}(\Lambda_m) - 
\left( \eta_{i-(k+1)}^{(k+1)}(\Lambda_m) \right)^
{\frac{r_{k+1}- r_k}{r_{k+2}-r_{k+1}}} = 0
\quad \textnormal{uniformly in} \quad i.
\end{equation}  
Let us first deal with the indices $i$ such that $i > k+1$, 
in this case, if we denote by  
$$
a_j = \left(1 - \frac{r_{k+2} - r_{k+1}}{r_j - r_{k+1}} \right)
\quad \textnormal{and} \quad
b_j = \left(1 - \frac{r_{k+1} - r_{k}}{r_j - r_{k}} \right); \quad j> k+2,
$$
then, according to (\ref{controlpoints}), and imitating the notations 
of the Lemma \ref{lorentzlemma}, we have 
\begin{equation*}
\eta_{i-k}^{(k)}(\Lambda_0) = \prod_{j=i+1}^{m} b_j = B_i(m)
\quad \textnormal{and} \quad
\eta_{i-(k+1)}^{(k+1)}(\Lambda_0) = \prod_{j=i+1}^{m} a_j =A_i(m).
\end{equation*}
The fact that the sequence $(r_k)_{0\leq k \leq \infty}$ is 
strictly increasing, shows that $a_j$ and $b_j$ are elements 
of the interval $]0,1[$ ($j> k+2$). Moreover, as 
$\lim_{j\to\infty} r_j = \infty$ and using the l'Hospital's rule,
show that 
$$
\lim_{j\to\infty} \frac{\ln b_j}{\ln a_j} =  
\frac{r_{k+1}-r_{k}}{r_{k+2}-r_{k+1}} = \gamma >0.
$$
To prove that for any fixed $i>k+1$, we have $\lim_{m\to\infty}A_i(m) = 0$, 
we proceed as follows : Since $a_j \in ]0,1[$, we have 
\begin{equation}\label{infinityc}
\frac{1}{a_j} \geq 1 + \frac{r_{k+2}-r_{k+1}}{r_{j}-r_{k+1}}
\quad \textnormal{and then} \quad
\frac{1}{A_i(m)} \geq  \prod_{j=i+1}^{m} 
\left( 1 + \frac{r_{k+2}-r_{k+1}}{r_{j}-r_{k+1}} \right).
\end{equation}
Moreover, 
$$
\sum_{j=i+1}^{m} \frac{r_{k+2}-r_{k+1}}{r_{j}-r_{k+1}}
= (r_{k+2}-r_{k+1}) \sum_{j=i+1}^{m} \frac{1}{r_{j}-r_{k+1}}
\geq  (r_{k+2}-r_{k+1}) \sum_{j=i+1}^{m} \frac{1}{r_{j}}.
$$
Therefore, the conditions (\ref{twoconditions}) show that 
$$
\lim_{m\to\infty} \sum_{j=i+1}^{m} \frac{r_{k+2}-r_{k+1}}
{r_{j}-r_{k+1}} = \infty 
\quad \textnormal{thus} \quad
\lim_{m\to\infty} \prod_{j=i+1}^{m} 
\left( 1 + \frac{r_{k+2}-r_{k+1}}{r_{j}-r_{k+1}} 
\right) = \infty,
$$
which by (\ref{infinityc}) conclude that $\lim_{m\to\infty}A_i(m) = 0$.
Similar treatments for $b_j$ show that for any fixed $i>k+1$,
$\lim_{m\to\infty}B_i(m) = 0$. Therefore, applying Lemma \ref{lorentzlemma}
(after an obvious shift of indices) shows the convergence of 
(\ref{convergenceterm}) uniformly in $ i > k+1$. For $i < k+1$,
then the term in (\ref{terms}) is zero and 
if $i = k+1$ the term in (\ref{terms}) is $\eta_{1}^{(k)}(\Lambda_m)$,
which converges to zero as $m$ goes to infinity. Then, using the 
trick of finitude as at the end of the proof of Lemma \ref{lorentzlemma},
conclude the proof. 
\end{proof}
%%%%%%%%%%%%%%%%%%%%%%%%%%%%%%%%%%%%%%%%%%%%%%%%%%%%%%%%%%%%%%%%%
In order to conclude the proof of the main Theorem \ref{muntzelevation}
using Theorem \ref{intermediate}, we need to show that the point set 
$D_{m} = \{\eta_i(\Lambda_m)^{1/r_1}, i=0,...,m \}$ 
form a dense subset of the interval $[0,1]$ as $m$ goes to infinity. 
For this aim, we need the following result proven by Hirschman and Widder 
\cite{hirschman} and Gelfond \cite{gelfond}. 

\begin{theorem}\label{gelfondtheorem}
{(\bf{Hirschman-Widder \cite{hirschman}, Gelfond \cite{gelfond}})}

\noindent Let $\Lambda_{\infty} = (0=r_0,r_1,...,r_n,....,r_l,...)$
be an infinite sequence of strictly increasing real numbers
such that 
\begin{equation*}
\lim_{s\to \infty} r_s = \infty
\quad \textnormal{and} \quad
\sum_{i=1}^{\infty} \frac{1}{r_i} = \infty.
\end{equation*}
To every continuous function $f$ in the interval $[0,1]$, we associate 
the operator $B_n^{\Lambda_\infty}(f)$ defined as 
\begin{equation*}
B_m^{\Lambda_\infty}(f)(x) = \sum_{i=0}^{m} 
f(\eta_{i}(\Lambda_m)^{1/r_1}) H^{m}_{i,\Lambda_m}(x),
\end{equation*}
where $\Lambda_m = (0=r_0,r_1,...,r_m)$. Then the sequences 
$B_m^{\Lambda_\infty}(f)$ is uniformly convergent with limit $f$
as $m$ goes to infinity.  
\end{theorem} 

Using the last theorem, we can now prove the following   
%%%%%%%%%%%%%%%%%%%%%%%%%%%%%%%%%%%%%%%%%%%%
%%%%%%%%%%%%%%%%%%%%%%%%%%%%%%%%%%%%%%%%%%%%
\begin{proposition}\label{densitytheorem}
Let $\Lambda_\infty = (0=r_0,r_1,...,r_n,....,r_l,...)$
be an infinite sequence of strictly increasing 
real numbers such that 
\begin{equation*}
\lim_{s\to \infty} r_s = \infty
\quad \textnormal{and} \quad
\sum_{i=1}^{\infty} \frac{1}{r_i} = \infty.
\end{equation*}
Denote by $D_m$ the point set 
$D_{m} = \{\eta_i(\Lambda_m)^{1/r_1}, i=0,...,m \}$. 
Then, as $m$ goes to infinity, the point set $D_m$ 
form a dense subset of the interval $[0,1]$.
\end{proposition}
%%%%%%%%%%%%%%%%%%%%%%%%%%%%%%%%%%%%%%%%%%%
%%%%%%%%%%%%%%%%%%%%%%%%%%%%%%%%%%%%%%%%%%%
\begin{proof}
Let $\epsilon$ be a strictly positive real number and let $x_0$ be a 
real number in the interval $[0,1]$. 
Consider the continuous piecewise linear function $f$ 
defined as 
\begin{equation*}
f(x) = \begin{cases}
& \frac{-x}{x_0} +1 \quad \textnormal{if} \quad  
0 \leq x \leq x_0 \\
& \frac{x}{1-x_0} - \frac{x_0}{1-x_0} \quad \textnormal{if} \quad 
x_0 \leq x \leq 1.
\end{cases}
\end{equation*}
As the function $f$ is continuous, then by Theorem \ref{gelfondtheorem},
there exist an  $m_0$ such that for any $m \geq m_0$, we have 
\begin{equation*}
|f(x_0) - \sum_{i=0}^{m} 
f(\eta_{i}(\Lambda_m)^{1/r_1}) H^{m}_{i,\Lambda_m}(x_0)| 
< \epsilon.
\end{equation*}
Evaluating $f$ in the last equation leads to 
\begin{equation}\label{epsilonequation}
\begin{split}
& \frac{1}{x_0} \sum_{\eta_{i}(\Lambda_m)^{1/r_1} \leq x_0}
(x_0 - \eta_{i}(\Lambda_m)^{1/r_1}) H^{m}_{i,\Lambda_m}(x_0) +  \\
&\frac{1}{1-x_0} \sum_{\eta_{i}(\Lambda_m)^{1/r_1} > x_0}
(\eta_{i}(\Lambda_m)^{1/r_1} - x_0) 
H^{m}_{i,\Lambda_m}(x_0) < \epsilon. \\
\end{split}
\end{equation}  
Now, if $|x_0 - \eta_i(\Lambda_m)^{1/r_1}| > \epsilon \;$ for $i=1,2,...,m$,
then, using the fact that $\sum_{i=0}^{m} H^{m}_{i,\Lambda_m}(x_0) = 1$, shows 
that the left hand side expression of (\ref{epsilonequation}) is also strictly
greater than $\epsilon$, leading to a contradiction. Therefore, for any 
$\epsilon >0$ and any $x_0 \in [0,1]$, there exists an $m_0$ such that for any 
$m \geq m_0$, there exists an $i \leq m$ such that 
$|x_0 - \eta_i(\Lambda_m)^{1/r_1}| < \epsilon$.    
\end{proof}
%%%%%%%%%%%%%%%%%%%%%%%%%%%%%%%%%%%%%%%%%%%%%%%%%%%%%%%%%%%
\noindent As this point, we are ready to prove Theorem \ref{muntzelevation}. 
\vskip 0.3 cm
\noindent {\bf{Proof of Theorem \ref{muntzelevation} :}}
Let us prove the theorem when the sequence 
$\Lambda_{\infty} = (0=r_0,r_1,...,r_n,....,r_l,...)$ of strictly 
increasing real numbers satisfy the conditions (\ref{twoconditions}).
In this case, if we denote by $P$ the Gelfond-B\'ezier curve 
with control points $(P_0,P_1,...,P_n)$, we have to show 
that given a point $x \in [0,1]$ and a sequence of real numbers 
$\eta_{i(x)}(\Lambda_m)^{1/r_1}$ that converges to $x$ as $m$ goes 
to infinity (this is possible thanks to the density proposition
\ref{densitytheorem}), the points $b^{m}_{i(x)}$ converges to 
$P(x)$ as $m$ goes to infinity. As the function $P$ is continuous, 
$P(\eta_{i(x)}(\Lambda_m)^{1/r_1})$ converges to $P(x)$.
Therefore, for any $\epsilon >0$, there exists an $M_0$, 
such that for any $m \geq M_0$, we have 
\begin{equation*}
|P(x) - P(\eta_{i(x)}(\Lambda_m)^{1/r_1})| < \epsilon.
\end{equation*}
Moreover, from Theorem \ref{intermediate}, there exists an $M_1$ such that for 
any $m \geq M_1$
\begin{equation*}
|b^{m}_{i(x)} - P(\eta_{i(x)}(\Lambda_m)^{1/r_1})| < \epsilon.
\end{equation*}
Therefore, $|P(x) -b^{m}_{i(x)}|< 2\epsilon$ for any $m \geq max(M_0,M_1)$,
thereby proving the pointwise convergence of the dimension elevated control polygons
to the Gelfond-B\'ezier curve. The convergence is uniform as we have 
\begin{equation*}
\max_{x} |P(x) - b^{m}_{i(x)}| \leq 
\max_x |P(x) - P(\eta_{i(x)}(\Lambda_m)^{1/r_1})| 
+ \max_{i(x)}  |P(\eta_{i(x)}(\Lambda_m)^{1/r_1})-b^{m}_{i(x)}|
\end{equation*}
The function $P$ is continuous in the compact interval $[0,1]$, thus 
$$
\max_x |P(x) - P(\eta_{i(x)}(\Lambda_m)^{1/r_1})| \rightarrow 0 
\quad \textnormal{as $m$ goes to infinity},
$$
and Theorem \ref{intermediate} shows that 
$$
\max_{i(x)}  |P(\eta_{i(x)}(\Lambda_m)^{1/r_1})-b^{m}_{i(x)}| 
\rightarrow 0 
\quad \textnormal{as $m$ goes to infinity}.
$$
This conclude the proof of the if part of the theorem.
To prove the only if part of the theorem, 
we proceed by contradiction. Let us assume that 
the real number $r_i$ satisfy 
\begin{equation*}
\sum_{i=1}^{\infty} \frac{1}{r_i} < \infty.
\end{equation*}
Without loss of generality, we can take the case in which 
the control points $(P_0,P_1,...,P_n)$ are real numbers such 
that $P_0 = 0$, $P_1 = 1$ and $P_1 < P_2 < ...< P_n$. 
In this case, for any $m$, we have $b^{m}_{0} = 0$ and $b^{m}_{1}$ 
is a strictly decreasing function of $m$ that converges to  
a strictly positive number $0<\delta <1$ 
\begin{equation*}
\delta = \lim_{m\to\infty} b^{m}_{1} = 
\prod_{i=2}^{\infty} \left( 1 - \frac{r_1}{r_{j}} \right).
\end{equation*}
In this case we would have a gap between the point zero and $\delta$, namely, 
if we take $x = \delta/2$ than for any $m \geq 0$, we have 
$$
|b^{m}_{i} - x| > \frac{1}{4} 
\quad \textnormal{for} \quad 
i = 1,...,m
$$
and therefore, the limiting control polygon does not converge 
pointwise to the Gelfond-B\'ezier curve.  
   
\section{Discussion}
In the following, we give a list of directions for future research as well as some 
open problems.
\vskip 0.3 cm 
%%%%%%
{\bf{1- }}The corner cutting scheme (\ref{initial}) and (\ref{cornercutting}) 
can be generalized so as to describe the dimension elevation algorithm 
of rational Gelfond-B\'ezier curves. In this case, weighted corner-cutting
schemes are derived and the methods developed in this 
work can contribute to the study of the convergence of 
these new family of corner cutting schemes.
\vskip 0.3 cm 
%%%%%%
{\bf{2- }}We can study the limiting polygon of the corner 
cutting scheme in case we relax the hypothesis of strictly increasing 
sequences $0 < r_1 <r_2 <...<r_m<...$ to only increasing sequences
$0 < r_1 \leq r_2 \leq ....\leq r_m \leq ...$. In this case, the 
Gelfond-B\'ezier curves involve logarithmic functions \cite{gelfond},
namely, if we rewrite the exponent $(r_1,r_2,...,r_n) = (\tilde{r}_1,
\tilde{r}_2,....,\tilde{r}_m)$ where the real number $\tilde{r}_i$ 
are distinct and if we denote by $m_j, j=1,...,m$ the number of indices 
$i = 0,...,n$ for which $r_i = \tilde{r}_i$, then the space 
$span(1,t^{r_1},t^{r_2},...,t^{r_n}) = span(1,x^{r_j}(\ln x)^{i}; j=0,1,..m; 
i= 0,1,...,m_j-1)$. The results of this work could be extended to this case 
by a limiting process.
\vskip 0.3 cm 
%%%%%%
{\bf{3- }}The Gelfond-B\'ezier curves are too ``degenerate'' at 
the origin to study the dimension elevation algorithm  
in case we impose no condition of monotonicity on the 
real numbers $r_i$. For instance, if we consider the case $n=3,$ 
$r_1=1, r_2 = 2, r_3 = 3$ and $r_j = 1/j$, for $j > 3$ and we start 
with a control polygon $(P_0,P_1,P_2,P_3)$ then the dimension elevated 
control polygon to the order $m$ is not obtained by a corner cutting 
scheme similar to (\ref{initial}) and (\ref{cornercutting}) but 
instead the algorithm collapses the first $m-3$ control points
to $P_0$ while the remaining control points are given by $(P_1,P_2,P_3)$
\cite{ait-haddou1}.
However, if we consider the dimension elevation algorithm of 
Gelfond-B\'ezier curves far from the origin, i.e. over an interval
$[a,1]$ with $a >0$, then the Gelfond-Bernstein basis coincide 
with the Chebyshev-Bernstein basis \cite{ait-haddou1},
the degeneracy at the origin disappear and the algorithm leads 
to a family of corner cutting schemes without imposing any condition
of monotonicity on the real numbers $r_i$. Unfortunately, such family 
of corner cutting schemes involves rather complicated 
coefficients expressed in term of Schur functions \cite{ait-haddou2}. 
It will be interesting to find, for the far from the origin case, 
conditions on the real number $r_i$ for the convergence of the 
dimension elevation algorithm to the underlying curve.
In the theory of \muntz spaces over an 
interval $[a,1]$ with $a>0$, and in which we impose no condition
on the real numbers $r_i$ (beside that they are pairwise distinct), 
then the corresponding \muntz space is a dense subset of $C([a,1])$
if and only if the real numbers $r_i$ satisfy the so-called 
full \muntz condition \cite{borwein2} 
\begin{equation}\label{fullmuntz}
\sum_{r_k \neq 0} \frac{1}{|r_k|} = \infty.
\end{equation}
The question is then does the surprising emergence of the \muntz 
condition in Theorem \ref{muntzelevation} for the real numbers 
$r_i$ with the condition of the theorem, repeat itself for the full 
\muntz condition (\ref{fullmuntz}) for the far from the origin case.
%%%%%%%%%%%%%%%%%%%%
\vskip 0.3 cm 
{\bf{4- }}It is not difficult to show that with the conditions of Theorem 
\ref{muntzelevation}, the conditions (\ref{twoconditions})
are sufficient for the uniform convergence of the dimension elevation algorithm 
of Chebyshev-B\'ezier curve of the associated \muntz space to the 
underlying curve (the proof will be published elsewhere).
However, the pointwise convergence is more involved and 
the question of whether the \muntz condition is necessary prove to be
interesting.  
%%%%%%%%%%%%%%%%%%%%
\vskip 0.3 cm 
{\bf{5- }}It is probably a difficult problem to study the rate 
of convergence of the corner  cutting scheme (\ref{initial})
and (\ref{cornercutting}) in case the real numbers $r_i$ satisfy
the condition of Theorem \ref{muntzelevation}.
Adapting the method of Prautzsch and  Kobbelt \cite{kobbelt}
to this problem shows for example that if the numbers $r_i$ are integers and that  
there exists a constant $K$ such that $r_j \leq K +j$ for all 
$j \geq 1$ then the rate of convergence of the corner cutting 
scheme is in $O(\frac{1}{m})$. 

%%%%%%%%%%%%%%%%%%%%
\vskip 0.3 cm 
{\bf{6- }}It may happen that studying the limiting polygon of the corner cutting 
scheme (1) and (2) is richer under the condition  
\begin{equation}\label{nomuntz}
\sum_{k=1}^{\infty} \frac{1}{r_k} < \infty,
\end{equation}     
in analogy with the problem of studying the uniform closure of the \muntz space 
$E_{\infty} = span(1,t^{r_1},...,t^{r_n},...)$ under the condition 
(\ref{nomuntz}).

\subsubsection*{Acknowledgment : This work was partially supported by the MEXT Global COE project.
Osaka University, Japan.}

\end{document}